\documentclass[12pt]{article}
\usepackage[margin=1in]{geometry}
\usepackage{verbatim}

\usepackage{algorithm,algorithmicx,algpseudocode}
\algdef{SE}[FORUNTIL]{for}{until}[1]{\algorithmicfor \ #1 \algorithmicdo}[1]{\algorithmicuntil\ #1}%

\usepackage{amsmath,amssymb,amsfonts,graphicx,amsthm,nicefrac,mathtools,bm}
\usepackage{hyperref}
\usepackage[numbers]{natbib}
\usepackage{bbm}
\usepackage{color}
\usepackage{xspace}
\newtheorem{theorem}{Theorem}

\newtheorem{lemma}{Lemma}
\newtheorem{proposition}{Proposition}

\theoremstyle{definition}

\theoremstyle{remark}

\makeatletter
\newtheorem*{rep@theorem}{\rep@title}
\newcommand{\newreptheorem}[2]{%
\newenvironment{rep#1}[1]{%
\def\rep@title{#2 \ref{##1}}%
\begin{rep@theorem}}%
{\end{rep@theorem}}}
\makeatother
\newreptheorem{theorem}{Theorem}
\newreptheorem{lemma}{Lemma}
\newreptheorem{corollary}{Corollary}
\newreptheorem{proposition}{Proposition}

\DeclareMathOperator{\cov}{Cov}
\DeclareMathOperator{\var}{Var}

\DeclareMathOperator*{\argmin}{arg\,min}

\newcommand{\iid}[0]{i.i.d.\xspace}

\newcommand{\One}[1]{{\mathbbm{1}}\left\{{#1}\right\}}
\newcommand{\inner}[2]{\langle{#1},{#2}\rangle} % Inner product
\newcommand{\biginner}[2]{\left\langle{#1},{#2}\right\rangle} % Inner product
\newcommand{\norm}[1]{\lVert{#1}\rVert}
\newcommand{\Norm}[1]{\left\lVert{#1}\right\rVert}
\newcommand{\PP}[1]{\mathbb{P}\left\{{#1}\right\}} % Probability
\newcommand{\EE}[1]{\mathbb{E}\left[{#1}\right]} % Expectation
\newcommand{\Ep}[2]{\mathbb{E}_{#1}\left[{#2}\right]}
\newcommand{\EEst}[2]{\mathbb{E}\left[{#1}\ \middle| \ {#2}\right]} % Conditional expectation
\newcommand{\VV}[1]{\var\left({#1}\right)} % Variance

\newcommand{\VVst}[2]{\var\left({#1}\ \middle|\ {#2}\right)}

\def\R{\mathbb{R}}
\def\Z{\mathbb{Z}}

\newcommand{\ident}{\mathbf{I}}
\newcommand{\ones}{\mathbf{1}}

\newcommand\independent{\protect\mathpalette{\protect\independenT}{\perp}}
\def\independenT#1#2{\mathrel{\rlap{$#1#2$}\mkern2mu{#1#2}}}

\newcommand{\iidsim}{\stackrel{\mathrm{iid}}{\sim}}

\newcommand{\eps}{\epsilon}
\title{Convergence guarantee for the sparse monotone single index model}
\author{Ran Dai\thanks{Department of Biostatistics, University of Nebraska Medical Center}, Hyebin Song\thanks{Department of Statistics, The Pennsylvania State University},
 Rina Foygel Barber\thanks{Department of Statistics, University of Chicago}, Garvesh Raskutti\thanks{Department of Statistics, University of Wisconsin-Madison}}

\newcommand{\X}{\mathbf{X}}
\newcommand{\Y}{\mathbf{Y}}
\renewcommand{\Z}{\mathbf{Z}}

\newcommand{\iso}{\textnormal{iso}}

\newcommand{\HT}{\Psi_s}
\newcommand{\prp}[1]{\mathcal{P}^{\perp}_{#1}}
\newcommand{\normal}{\mathcal{N}}

\usepackage{natbib}

\begin{document}
\maketitle

\abstract
% Enter Abstract here
 We consider a high-dimensional monotone single index model (hdSIM), which is a semiparametric extension of a high-dimensional generalize linear model (hdGLM), where the link function is unknown, but constrained with monotone and non-decreasing shape. We develop a scalable projection-based iterative approach, the ``Sparse Orthogonal Descent Single-Index Model" (SOD-SIM),  which alternates between sparse-thresholded orthogonalized ``gradient-like" steps and isotonic regression steps to recover the coefficient vector. Our main contribution is that we provide finite sample estimation bounds for both the coefficient vector and the link function in high-dimensional settings under very mild assumptions on the design matrix $\X$, the error term $\epsilon$, and their dependence. The convergence rate for the link function matched the low-dimensional isotonic regression minimax rate up to some poly-log terms ($n^{-1/3}$). The convergence rate for the coefficients is also $n^{-1/3}$ up to some poly-log terms. This method can be applied to many real data problems, including GLMs with misspecified link, classification with mislabeled data, and classification with positive-unlabeled (PU) data. We study the performance of this method via both numerical studies and also an application on a rocker protein sequence data.  
%under some mild assumption for the link function (smooth and strictly increasing) matches the low dimensional parametric rate up to some poly-log terms ($n^{-1/2}$). 

\section{Introduction}
Single index models (SIMs) provide a semi-parametric extension of linear models, where a scalar response variable $Y \in \R$ is related to a predictor vector $X \in \R^p$ via 
\begin{equation}\label{def:SIM}
\mathbb{E}[Y|X] = g_\star (X^\top u_\star)
\end{equation}
for some unknown parameter vector $u_\star$ and unknown link function $g_\star$. 

In this paper, we consider a shape-constrained single index model \eqref{def:SIM} where $g_\star$ belongs to a class of monotonic (but not necessarily smooth) functions and the parameter $u_\star$ is high-dimensional and sparse (with sparsity level $s_\star$). We develop a scalable algorithm and provide theoretical guarantees for high-dimensional single index models. Prior theoretical guarantees for the high-dimensional single index model rely on the distribution of $X$ being known or symmetric (see e.g.~\cite{plan2017,yang2017a}). In a number of settings, such assumptions on $X$ are not satisfied (in section \ref{example_PU} we show $X$ follows a non-symmetric mixture distribution) and prior theoretical guarantees do not apply. In this paper, we provide a scalable algorithm with theoretical guarantees for the high-dimensional single index model under more general design assumptions that include non-symmetric distributions.

We consider a general case where mild assumptions are made on the design matrix $X$, the error $\epsilon := Y-\EEst{Y}{X}$, and their interaction. In particular, $X$ can be deterministic or random. If $X$ is random, the distribution of $X$ can be asymmetric. The distribution of $\epsilon$ may depend on $X$. In addition we assume that the unknown vector $u_\star$ is sparse and high-dimensional. As $g_\star$ is not assumed to be known a priori, the model \eqref{def:SIM} provides a more flexible model specification than a parametric model, while still circumventing the curse-of-dimensionality by avoiding a $p$-dimensional non-parametric function estimation. Two estimation problems exist in single index model framework: estimation of the unknown vector $u_\star$ and of the 1-dimensional regression function $g_\star$. Both are addressed in this paper.

To begin, we provide three motivating examples for the hdSIMs in~\eqref{def:SIM}. We see that each example reduces to the estimation problem of $u_\star$ or $g_\star$, or both under the model \eqref{def:SIM}.

\subsection{Examples}

\subsubsection{Example 1:  Mis-specified generalized linear models (GLMs)}
Generalized linear models (GLMs) are parametric extensions of linear models, which includes many popular models such as logistic and Poisson regression. GLMs relate the conditional expectation of $Y$ given $X$ via a link function $\psi$, such that $\EEst{Y}{X} = \psi^{-1}(X^{\top}u_\star)$, where a link function $\psi$ is assumed to be known. The parameter of interest $u_\star$ is usually estimated via an iterative procedure using the knowledge of $\psi$; therefore, when $\psi$ is misspecified, an output is inevitably biased for $u_\star$, even in an asymptotic sense. Using the single index model framework, we produce an estimate of $u_\star$ without particular link function specification. 

\subsubsection{Example 2:  Classification with corrupted labels}
Another interesting example involves classification with corrupted labels. In particular, instead of observing a sample $(X,\tilde{Y}) \in \mathbb{R}^p \times \{0,1\}$ where $\tilde{Y}$ may follow, e.g., a logistic model, we observe the pair $(X,Y) \in \mathbb{R}^p \times \{0,1\}$ where $Y$ is a systematically corrupted version of $\tilde{Y}$. In particular $\tilde{Y}$ is related to $X$ via $\mathbb{E}[\tilde{Y}|X] = \tilde{g}(X^{\top}u_\star)$ for some unknown parameter vector $u_\star \in \mathbb{R}^p$ and a monotonic function $\tilde{g}$, which is potentially unknown. 

We assume the corrupted sample $Y$ is independent of $X$ given $\tilde{Y}$, which corresponds to a missing-completely-at-random assumption. The corruption structure is completely specified by  corruption probabilities $\rho_1 = \mathbb{P}(Y=0|\tilde{Y}=1)$ (the probability that a positive label is corrupted) and $\rho_0 = \mathbb{P}(Y=1|\tilde{Y}=0)$ (the probability that a negative label is corrupted), assuming $\rho_0+\rho_1<1$ for identifiability of the model. Under these assumptions, $\mathbb{E}[Y|X] = \mathbb{P}(Y=1|X)$ is given by
\begin{align}\label{intro:example2:gstar}
\mathbb{P}(Y=1|X) &= \tilde{g}(X^{\top} u_\star)\mathbb{P}(Y=1|\tilde{Y}=1) + (1-\tilde{g}(X^{\top} u_\star))\mathbb{P}(Y=1|\tilde{Y}=0) \nonumber\\
&= \tilde{g}(X^{\top}u_\star)(1-\rho_1) + (1-\tilde{g}(X^{\top}u_\star))\rho_0\nonumber\\
& = g_\star(X^\top u_\star)
\end{align}
where we let $g_\star (s) :=  \tilde{g}(s)(1-\rho_1)+ (1-\tilde{g}(s))\rho_0$. (Note $g_\star$ is monotonic since $\tilde{g}$ is monotonic and $\rho_0+\rho_1<1$ by assumption.)

%a $2 \times 2$ confounding matrix $Q$ such that $Q_{k,l} = \mathbb{P}(Y =k |\tilde{Y} = l) $, whose off-diagonal terms are smaller than diagonals. Under these assumptions, $\mathbb{E}[Y|X] = \mathbb{P}(Y=1|X)$ is given by
%\begin{align}\label{intro:example2:gstar}
%\mathbb{P}(Y=1|X) &= g(X^{\top} u_\star)\mathbb{P}(Y=1|\tilde{Y}=1) + (1-g(X^{\top} u_\star))\mathbb{P}(Y=1|\tilde{Y}=0) \nonumber\\
%&= g(X^{\top}u_\star)Q_{11} + (1-g(X^{\top}u_\star))Q_{10}\nonumber\\
%& = g_\star(X^\top u_\star)
%\end{align}
%where we let $g_\star (s) :=  g(s)Q_{11} + (1-g(s))Q_{10}$. Note $g_\star$ is monotonic since $g_\star (s) = g(s) (Q_{11} - Q_{10}) +Q_{10}$ and we assume that $Q_{11}>Q_{10}$.
 In many practical situations the corruption probabilities $\rho_0$ and $\rho_1$ may be unknown. In such cases, 
even with the assumption of a specific $\tilde{g}$ (such as a sigmoid function $\tilde{g}(u) = 1/(1+\exp(-u))$ in the case of a logistic model for $\tilde{Y}$), with the unknown parameters $\rho_0$ and $\rho_1$, the new link function $g_\star$ is still unknown, which motivates moving beyond the parametric model setting.

%Estimation of $\mathbb{E}[\tilde{Y}|X] $ involves estimation of both $g_\star$ and $u_\star$, as well as $Q$ and a marginal probability $\mathbb{P}(\tilde{Y}=1) $ (Estimation of $g_\star$ and $u_\star$ may not be enough to build a classifier). Specifically,
%\begin{align*}
%\mathbb{E}[\tilde{Y}|X] = g_\star(X^\top u_\star) \mathbb{P}(\tilde{Y}=1|Y=1)+(1-g_\star(X^\top u_\star) ) \mathbb{P}(\tilde{Y}=1|Y=0)
%\end{align*}
%\[\mathbb{P}(\tilde{Y}=1|Y=1) = \frac{Q_{22} \mathbb{P}(\tilde{Y}=1) }{\mathbb{P}(Y=1) }= \frac{Q_{22} \mathbb{P}(\tilde{Y}=1) }{Q_{22} \mathbb{P}(\tilde{Y}=1) +Q_{21} \mathbb{P}(\tilde{Y}=0) }\]
%Likewise,
%\[\mathbb{P}(\tilde{Y}=1|Y=0) = \frac{Q_{12} \mathbb{P}(\tilde{Y}=1) }{\mathbb{P}(Y=0) }= \frac{Q_{12} \mathbb{P}(\tilde{Y}=1) }{Q_{12} \mathbb{P}(\tilde{Y}=1) +Q_{11} \mathbb{P}(\tilde{Y}=0) }\]

\subsubsection{Example 3: Classification with positive-unlabeled (PU) data and unknown prevalence $\pi$}\label{example_PU}
In various applications such as ecology, text-mining, and biochemistry, negative samples are often hard or even impossible to be obtained. In positive-unlabeled (PU) learning, we use positive and unlabeled samples, where positive samples are taken from the sub-population where responses are known to be positive, and unlabeled samples are random samples from the population. 

To make this concrete, we begin with a distribution $\mathcal{D}$ on data $(\tilde{X},\tilde{Y})\in \mathbb{R}^p \times \{0,1\}$. Here $\tilde{X}$ is a feature vector while $\tilde{Y}\in\{0,1\}$ is a true label, ``positive'' (1) or ``negative'' (0). In some applications, it is not possible to observe (negatively) labeled data; instead, we observe unlabeled data (i.e., a feature vector $X$, drawn from the marginal distribution of $\tilde{X}$ under the joint distribution $\mathcal{D}$), and a subsample of positively labeled data (i.e., a feature vector $X$ drawn from the conditional distribution of $\tilde{X}\mid \tilde{Y}=1$, derived from the joint distribution $\mathcal{D}$). The new ``label'' $Y\in\{0,1\}$ then specifies whether $X$ was drawn as an unlabeled sample ($Y=0$) or as a positive-only draw ($Y=1$). Writing $\gamma$ as the ratio of positive-only samples to unlabeled samples, the conditional distribution $Y\mid X$ can be calculated as
\[\mathbb{P}(Y=1\mid X=x) 
= \frac{\gamma \pi^{-1}}{\frac{1}{\mathbb{P}(\tilde{Y}=1\mid \tilde{X}=x)} + \gamma \pi^{-1}},\]
where $\pi = \mathbb{P}(\tilde{Y}=1)$ is the proportion of positive labels under the original distribution $\mathcal{D}$ on $(X,Y)$.
Therefore, we can see that if the distribution of $\tilde{Y}\mid\tilde{X}$, in the underlying population, satisfies the monotone single-index model, then the same is true for the new PU data---specifically, if $\mathbb{P}(\tilde{Y}=1\mid \tilde{X}=x)= \tilde{g}(x^\top u_\star)$
for some monotone function $\tilde{g}$ and some vector $u_\star$, then
\[\mathbb{P}(Y=1\mid X=x)  = g_\star(x^\top u_\star)\]
for 
\[g_\star(v) = \frac{\gamma \pi^{-1}}{\frac{1}{\tilde{g}(v)} + \gamma \pi^{-1}},\]
which is also a monotone function, inheriting this property from $\tilde{g}$.

In \cite{song2020}, this problem is addressed with a parametric approach, by assuming that $\tilde{g}$ is known (commonly the sigmoid function, i.e., assuming a logistic model for $\tilde{Y}\mid\tilde{X}$), and assuming that the proportion $\pi = \mathbb{P}(\tilde{Y}=1)$ of positive labels in the underlying population is known as well. In settings where $\pi$ is unknown, however, this approach can no longer be applied, even if $\tilde{g}$ is known; reliable estimation of such proportion is very challenging and heavily relies on model specification \citep{hastie2013} and can lead to unreliable inference.
In contrast, assuming only that $g_\star$ is a monotone function allows for both unknown $\pi$ and unknown (monotone) $\tilde{g}$.

Another interesting aspect of this PU problem is that the distribution of $X$ is, in general, not symmetric. Even in a setting where the distribution of features in the population (i.e., the marginal distribution of $\tilde{X}$) can plausibly be assumed to be symmetric, this will no longer be the case for $X$---this is because the marginal distribution of $X$ is a mixture between the marginal distribution of $\tilde{X}$, and the conditional distribution of $\tilde{X}\mid\tilde{Y}=1$. Therefore, this setting cannot be addressed with existing methods that rely on a symmetric design assumption.

\subsection{Prior work}

The single index model has been an active area in statistics for a number of decades, since its introduction in econometrics and statistics \cite{ichimura1993, horowitz1996,klein1993}. There is extensive literature about estimation of the index vector $u_\star$, which can be broadly classified into two categories: methods which directly estimate the index vector while avoiding estimation of a ``nuisance'' infinite-dimensional function, and methods which involve estimation of both the infinite-dimensional function and the index vector. 

Methods of the first type usually require strong design assumption such as Gaussian or an elliptically symmetric distribution. \citet{duan1991} proposed an estimator of $u_\star$ using sliced inverse regression in the context of sufficient dimension reduction. $\sqrt{n}$ consistency and asymptotically normality of the estimator are established in the low-dimensional regime, under the elliptically symmetrical design assumption. This result is extended in \citet{lin2018} in high-dimensional regime under similar design assumptions. \citet{plan2017} studied sparse recovery of the index vector $u_\star$  when $X$ is Gaussian, and established a non-asymptotic $\mathcal{O}((s\log (2p/s) /n)^{1/2}) $ mean-squared error bound, where $s$ is a sparsity level, i.e. $s := \|u_\star\|_0$. \citet{ai2014} studied the error bound of the same estimator when the design is not Gaussian and showed that the estimator is biased and the bias depends on the size of $\|u_\star\|_\infty$. The proposed estimator in \citet{plan2017} is generalized in several directions; \citet{yang2017b} propose an estimator based on the score function of the covariate, which relies on prior knowledge of covariate distribution. \citet{chen2017} proposed estimators based on U-statistics which also include \citet{plan2017} as a special case with standard Gaussian design assumption. 

Methods of the second type usually do not require strong design assumption, and our method also falls into this category. One popular approach is via M-estimation or solving estimation equations. Under smoothness assumptions of $g_\star$, methods for a link function estimation were proposed via Kernel estimation \cite{ichimura1993, hardle1993,klein1993,delecroix2003,cui2011}, local polynomials \cite{carroll1997,chiou1998}, or splines \cite{yu2002,kuchibhotla2017}. An estimator for the index vector is then obtained by minimizing certain criterion function such as squared loss \cite{ichimura1993,hardle1993,yu2002} or solving an estimating equation \cite{chiou1998,cui2011}. $\sqrt{n}$ consistency and asymptotic normality of those proposed estimators were established for those estimators in the low-dimensional regime. Another approach is the average derivative method \cite{powell1989, hristache2001}, which takes advantage of the fact $\frac{d}{dx} \EEst{Y}{X=x} = u_\star g'(x^\top u_\star)$. An estimator is then obtained through a non-parametric estimation of $\frac{d}{dx} \EEst{Y}{X=x}$, also under smoothness assumption of $g_\star$. There are also studies for single index model in high-dimensional regime, using PAC-Bayesian approach \cite{alquier2013} or incorporating penalties \cite{wang2012, radchenko2015}.  

% proposed an algorithm based on reversible jump MCMC and an oracle inequality for the prediction error is presented. \cite{Radchenko2015-wa} proposed a path-fitting algorithm based on a spline method under smooth link assumption with an asymptotic guarantee. 

There have been an increasing number of studies where the link function is restricted to have a particular shape, such as a monotone shape, but is not necessarily a smooth function. 
% Estimation of an index vector $u_\star$, under the monotonic the link assumption first appear in \cite{Han1987-cc}, where non-parametric maximum rank correlation estimator was proposed and strong consistency of the estimator was established.
\citet{kalai2009} and \citet{kakade2011} investigated single index problem in low-dimensional regime under monotonic and Lipschitz continuous link assumption, and proposed iterative perceptron-type algorithms with prediction error guarantees. In particular, an isotonic regression is used for the estimation of $g_\star$ then update $u_\star$ based on estimated $g_\star$. \citet{ganti2017} extended \cite{kalai2009,kakade2011} to the high-dimensional setting via incorporating a projection step onto a sparse subset of the parameter space. They empirically demonstrated good predictive performance of their algorithms, but no theoretical guarantee were provided. \citet{balabdaoui2019} study the least square estimators under monotonicity constraint and establish $n^{-1/3}$ consistency of the least-square estimator. Also in \citet{balabdaoui2018}, assuming a smooth link function $g_\star$, a $\sqrt{n}$ consistent estimator for the index vector $u_\star$, based on solving a score function, is proposed. 

Our work focuses on estimation of the index vector in high-dimensional regime where an unknown function is assumed to be Lipschitz-continuous and monotonic. We provide an efficient algorithm via iterative hard-thresholding and a non-asymptotic $\ell_2$ and mean function error bound. Unlike \citet{ganti2017} and \citet{kakade2011}, our algorithm does not require to input a Lipschitz constant of the unknown function. 

\subsection{Our contributions}

Our major contributions can be summarized as follows:
\begin{itemize}
\item Develop a scalable projection-based iterative approach,  the ``Sparse Orthogonal Descent Single-Index Model" (SOD-SIM) algorithm, which alternates between sparse-thresholded orthogonalized ``gradient-like" steps and isotonic regression steps to recover the coefficient vector.

\item Provide finite sample convergence guarantees for the SOD-SIM algorithm, both in estimating the parameter vector $u_\star$ and the mean function $g_\star(X^\top u_\star)$ that the error scales as $\mathcal{O}\big(\frac{s_\star^{1/2} \log np}{n^{1/3}}\big)$, where $s_\star$ is the sparsity level of $u_\star$. Note that the minimax rate for low-dimensional isotonic regression is $n^{-1/3}$ and our algorithm achieves this rate up to log factors. This rate also matches the $n^{-1/3}$ rate shown for least square estimator in \cite{balabdaoui2019}. However, the algorithm to obtain the least square estimator is computationally intensive even in low dimensional cases. Whereas in our work, we simultaneously showed the statistical and algorithmic convergence of a computationally efficient algorithm, which can also handle high dimensional case. In estimating $u_\star$, with $X$ being elliptically symmetrically distributed, the link-free slicing regression estimates by \citet{duan1991} has been shown to be $\sqrt{n}$ consistent with asymptotic normal distribution. When $g_\star$ is bounded and continuously differentiable, the score estimates proposed by \citet{balabdaoui2018} has been shown to be $\sqrt{n}$ consistent with asymptotic normal distribution. In comparison, our convergence result is not asymptotic and has milder assumptions on $X$ and $g_\star$; our algorithm is also much more computationally efficient than the score estimator.   
%\item Since our goal is largely to estimate the parameter vector $u_\star$, we also show that under additional assumptions including a $\beta$-min condition, we show that a one-step correction achieves the parametric low-dimensional $\frac{1}{\sqrt{n}}$ rate. The one-step correction is motivated by the earlier work of~\cite{balabdaoui2018} which proves $\frac{1}{\sqrt{n}}$ rates in the low-dimensional case.
\item Finally we provide empirical study results based on simulated and real data, which support our theoretical findings and also shows that our algorithm performs well compared with existing approaches.
\end{itemize}

%The remainder of the paper is organized as follows: Section~\ref{SecProb} provides a problem setup and we develop our algorithm; in Section~\ref{SecResults} we provide theoretical guarantees for our algorithm which are supported by experimental results in Section~\ref{SecExp}. Proofs are provided in Section~\ref{SecProof}.

\section{Main results}

Let $X_1,\dots,X_n\in\R^p$ be the feature vectors and $Y_1,\dots,Y_n\in\R$ the response values,\footnote{We denote $Y_i$, $i=1,\cdots,n$ as continuous variables for simplicity. In practice, they can be in other data types (continuous/categorical/mixed). } which we model as
\[Y_i = g_\star(X_i^\top u_\star) + Z_i, \]
where $g_\star$ is monotone non-decreasing; $u_\star$ is a $s_\star$-sparse unit vector; and $Z_i\in\R$ denotes the noise term.

We write $\X\in\R^{n\times p}$ to denote the matrix with rows $X_i, i \in \{1,\cdots,n\}$, and $\Y\in\R^n$ and $\Z\in\R^n$ to denote the vectors with entries $Y_i$ and $Z_i$, respectively. For any function $g:\R\rightarrow\R$ and any $v\in\R^p$, we write $g(\X v)$ to mean that $g$ is applied elementwise to the vector $\X v$, i.e. $g(\X v) = \big(g(X_1^\top v),\dots,g(X_n^\top v)\big)$. 

In the remaining of this section, we first present our SOD-SIM algorithm. Then we present the finite sample convergence results for the estimation of $u_\star$ and $g_\star$.

\subsection{Algorithm}\label{sec:Alg}

Before defining our algorithm, we first define notations for the hard-thresholding operator $\HT(\cdot)$ and the isotonic regression operator $\iso(\cdot)$. First let $\HT:\R^p\rightarrow\R^p$ be the ``hard-thresholding operator'' at sparsity level $s$, i.e., 
\[ \HT(v)_i = \begin{cases} v_i \quad i \in S, \\ 0 \quad i \not\in S, \end{cases}\]
where $S\subseteq \{1, \cdots, p\}$ indexes the $s$ largest-magnitude entries of $v$ (ties may be broken with any mechanism).
Next, for any vectors $u\in\R^p$ and $v\in\R^n$, define
\[\iso_{\X u}(v) = \argmin_{w\in\R^n}\left\{\norm{v-w}^2_2 : w_i\leq w_j\text{ whenever }(\X u)_i \leq (\X u)_j\right\},\]
i.e.~the isotonic regression of $v$ onto $\X u$. This convex problem can be solved efficiently using pool-adjacent-violators algorithm (PAVA) \cite{deLeeuw2009}. 

With these definitions in place, we are ready to define our algorithm for solving the sparse single index model.\footnote{Implicitly, the initialization step in our algorithm assumes that $\X^\top(\Y - \bar{Y}\mathbf{1}_n)\neq 0$. In our proofs, we will verify that this holds with high probability.} Given a target sparsity level $s\geq 1$ and a step size $\eta>0$, the algorithm alternates between taking an orthogonalized gradient-like step (with step size $\eta$), and hard-thresholding to enforce $s$-sparsity. The steps of the algorithm are shown in Algorithm~\ref{alg}.

\begin{algorithm}[t]
\caption{Sparse Orthogonal Descent Single-Index Model (SOD-SIM) }
\label{alg}
\begin{algorithmic}
\State \textbf{Initialize:} 
\begin{equation}\label{eqn:u_0}u_0=\frac{\HT\big(\X^\top(\Y - \bar{Y}\mathbf{1}_n)\big)}{\norm{\HT\big(\X^\top (\Y - \bar{Y}\mathbf{1}_n)\big)}_2},\end{equation}
where $\bar{Y} = \frac{1}{n}\sum_i Y_i$ and $\mathbf{1}_n$ is the vector of $1$'s.
\for{$t=1,2,\dots$} 
\State $\bullet$ Compute $\iso_{\X u_{t-1}}(\Y)$, the isotonic regression of $\Y$ onto $\X u_{t-1}$, using PAVA.
\State $\bullet$ Take an orthogonal gradient-like step, 
\[\tilde{u}_t = u_{t-1} + \eta \cdot \prp{u_{t-1}}\left(n^{-1}\X^\top \big(\Y - \iso_{\X u_{t-1}}(\Y)\big)\right).\]
\State $\bullet$  Enforce sparsity and unit norm,
\[u_t = \frac{\HT(\tilde{u}_t)}{\norm{\HT(\tilde{u}_t)}_2}.\]
\until{some convergence criterion is reached.}
\end{algorithmic}
\end{algorithm}

%The initialization step estimator $u_0$ is a hard-thresholded slicing regression estimator; we will see later on that this initialization does not need to be accurate, but needs only to satisfy $\inner{u_0}{u_\star}$, i.e., not worse than random. Then we iteratively update the estimation of $g_\star$ and $u_\star$. The $g_\star$ estimation is through PAVA and $u_\star$ estimation is through a projected orthogonal gradient step, where the gradient is projected to the orthogonal space and the sparsity was enforced by hard-thresholding.
For our theoretical guarantee to hold, we will see that the target sparsity level $s$ needs to be  sufficiently large relative to the true sparsity $s_\star$ of $u_\star$. Under mild assumptions, our theory will guarantee that the iterations of this algorithm converge to $u_\star$ up to an error level that is $\mathcal{O}(s^{1/2}\log(np)/n^{1/3})$.

The iterative framework of this algorithm is similar to the CSI method of \citet{ganti2017}. A key difference is that, in the $g_\star$ estimation step, the CSI method enforces a Lipschitz constraint in addition to the monotonicity constraint, which requires choosing the Lipschitz constant in advance; whereas our method only uses isotonic regression to enforce monotonicity. In addition, no theoretical guarantee is given for CSI. Our gradient update step also has some similarities to the least squares method of \citet{balabdaoui2019}; however, their algorithm works explicitly to minimize an objective function $\norm{\Y - \iso_{\X u}(\Y)}^2_2$, while our SOD-SIM algorithm recovers $u_\star$ through an iterative procedure. The descent direction is not the gradient of least square objective exactly, but a modified version for computational simplicity. The orthogonal projection step enforce a steady ``gradient-like" decending rate. Empirically in many cases the algorithm performance is similar without the orthogonal projection step; but in some cases when the initialization $\inner{u_0}{u_\star}$ is close to $0$, the orthogonal projected algorithm has better performance. The orthogonal projection step is also important in solving the technical issue from the normalization step in the proof.

\subsection{Convergence guarantees for fixed design $\X$}
We first provide upper bounds of convergence rates for the estimation of $u_\star$ and $g_\star$ when $\X$ is fixed. We begin with our assumptions. 
 
\subsubsection{Assumptions}\label{sec:assumptions}
We assume that
\begin{equation}\label{eqn:model}
\Y = g_\star(\X u_\star) + \Z,~\text{where}~ \X\in \R^{n\times p}, \Y \in \R^n, \Z\in \R^n, \end{equation}
and we place the following assumptions on the underlying function $g_\star$, parameters $u_\star$, design matrix $\X$ and noise $\Z$.

\paragraph{Assumptions on the signal}
We assume $g_\star$ is a bounded, Lipschitz, and monotone non-decreasing function, i.e. for $t_1>t_0$,
\begin{equation}\label{eqn:lipschitz} 0 \leq g_\star(t_1) - g_\star(t_0) \leq L (t_1 - t_0),\textnormal{ and }g_\star(\cdot): \R\rightarrow [-B,B],\end{equation}
while the true parameter vector $u_\star$ is a sparse unit vector,
\begin{equation}\label{eqn:sparsity}\norm{u_\star}_2=1\textnormal{ and }
\left|\textnormal{Support}(u_\star)\right|\leq s_\star.\end{equation}

For the design matrix $\X$, we assume  upper-bounded sparse eigenvalues,
\begin{equation}
\label{eqn:sparse_eig}\frac{1}{n}\norm{\X u}^2_2\leq \beta\norm{u}^2_2 ~ \textnormal{ for all $(2s+s_\star)$-sparse $u\in\R^p$.}
\end{equation}
For the corresponding lower bound, we require an identifiability condition, which is slightly stronger than the usual restricted eigenvalue type condition. This is because we are working with a substantially larger class of models than the usual parametric regression setting. In particular, there exists $\eps_n >0$ such that,
\begin{multline}\label{eqn:ID}\textnormal{for any $s$-sparse unit vector $u$ with $\norm{u-u_\star}_2\geq \eps_n$}\\\textnormal{and any monotone non-decreasing $g$,\quad }\frac{1}{L^2n}\norm{g(\X u) - g_\star(\X u_\star)}^2_2 \geq \alpha\norm{u-u_\star}^2_2. \end{multline}  
This condition effectively 
ensures that we cannot reproduce the true regression function $g_\star(\X u_\star)$ with some other sparse vector $u\neq u_\star$ without any loss in accuracy, unless $\norm{u-u_\star}_2$ is very small.
In particular, comparing assumptions~\eqref{eqn:lipschitz},~\eqref{eqn:sparse_eig}, and~\eqref{eqn:ID}, we can see that we must have
$\alpha \leq \beta$. To help interpret the parameters in this assumption, we should think of $\alpha$ and $\beta$ as constants (expressing properties of the function $g_\star$ that are constant regardless of sample size), while $\eps_n$ is vanishing as $n\rightarrow \infty$ (i.e., the identifiability condition will hold for vectors $u$ increasingly close to $u_\star$, as sample size $n$ increases). Later in Section \ref{sec:conv_randX}, we will show that \eqref{eqn:sparse_eig} and \eqref{eqn:ID} are satisfied with high probability when $X$ is a mixture of Gaussian distributions.
%i.e. 
%\begin{equation*}g_\star(t_0) - g_\star(t_1) \geq L_0 (t_0 - t_1) ~\text{with $t_0 > t_1$ and $L_0>0$} \end{equation*}

\paragraph{Assumptions on the noise}
For $\Z$, we assume for all $1\leq i\leq j\leq n$, $Z_i$'s are independent, with
\begin{multline}\label{eqn:noise_assump}
\EE{Z_i}=0~
\textnormal{and $\EE{e^{t Z_i}}\leq e^{t^2\sigma^2/2}$ for all $t \in \mathbb{R}$,}\\ \text{and}~ \sigma_{\max}^2\geq \VV{Z_i} = \sigma^2_i \geq\sigma^2_{\min}>0\textnormal{ and }
\left|\sigma_i - \sigma_j\right| \leq L_\sigma\cdot \frac{j-i}{n}.
\end{multline}

%The parameters in our assumptions will be required to satisfy the following relation: for $\delta > 0$, there exist universal constants $C_1$ and $C_2$ that
%\begin{equation}\label{eqn:s_condition}
%s \geq  s_\star \cdot \min\left\{ \left(\frac{\alpha}{L^2\beta} -\delta- C_1 \cdot \frac{\sqrt{s_\star\log n \log p}}{n^{1/3}}\right)^2, \left( \frac{4\alpha}{L} - C_2 \cdot \sqrt{\frac{s_\star \log p}{n}} \right)^2\right\}^{-1}.
%\end{equation}
%That is, the true sparsity $s_\star$ must be lower than the sparsity constraint $s$, 
%and the ratio between these two sparsity levels is controlled by the ratio $\beta/\alpha$, which acts somewhat like a ``condition number'' for the design matrix $\X$. %In particular, the term $\left(\frac{\alpha}{L^2\beta} - o(1)\right)^2$ is needed for the deterministic iterative guarantee Lemma \ref{lem:iter_update_deterministic} and the $\left( \frac{4\alpha}{L} - o(1) \right)^2$ term is from the initialization condition in Lemma \ref{lem:init}.

%More concretely, we require that for some $\delta > 0$,
%\begin{equation}\label{eqn:condition_num}
%\text{Remainder} \leq 2 \left(\frac{\alpha}{L^2\beta}-\delta - \sqrt{s_\star/s}  \right),\end{equation}
%where \text{Remainder} is a vanishing term defined in Lemma \ref{lem:iter_update_deterministic}. 
%And the $\left( \frac{4\alpha}{L} - o(1) \right)^2$ term is from the initialization condition in Lemma \ref{lem:init}.

\subsubsection{Convergence guarantee for the estimation of $u_\star$}

\begin{theorem}\label{thm:converge_main}

Assume the conditions~\eqref{eqn:model},~\eqref{eqn:lipschitz},~\eqref{eqn:sparsity}, \eqref{eqn:sparse_eig}, and~\eqref{eqn:ID} hold for the signal,
and condition~\eqref{eqn:noise_assump} holds for the noise.
Assume we run the algorithm with step size $\eta > 0$ and working sparsity level $s\geq 1$ satisfying
\begin{equation}\label{eqn:eta_s}  \eta \leq \frac{1}{L\beta}, \quad \frac{s}{s_\star} \geq c_s > \left(\frac{1}{\alpha\eta L}\right)^2.\end{equation}
For $\delta >0$, assume furthermore that
\begin{equation}\label{eqn:s_star_growth}\frac{s_\star\log(4p/\delta)}{n}\leq \frac{\sigma\alpha^2 L^2}{2\beta}.\end{equation}
Then, with probability at least $1-\delta$, it holds for all $t\geq 0$ that
\begin{equation}\label{eqn:u_bound}
\norm{u_t-u_\star}_2\leq  \sqrt{2}\cdot r^t + C\left(\eps_n + \frac{s^{1/2}\log(np/\delta)}{n^{1/3}}\right),
\end{equation}
where
\[r = \left(\frac{1-\alpha\eta L}{1-\sqrt{s_\star/s}}\right)^{1/2} < 1,\]
and where the constant $C$ (specified in the proof)
depends on $\alpha,\beta,L,B,\sigma,\eta, c_s$, but not on $n,p,s,\delta$ and $\epsilon_n$.

%Assume the conditions~\eqref{eqn:model},~\eqref{eqn:sparse_eig}, and~\eqref{eqn:ID} hold for the true regression function $g_\star$, vector $u_\star$,
%and design matrix $\X$, and that the subgaussianity assumption~\eqref{eqn:noise_assump} holds for the noise $\Z$. Suppose that $n$ is sufficiently large, $s_\star$ and $p$ satisfy $s_\star \log p = o(n^{2/3})$ and with some $\delta>0$ the working sparsity level $s$ satisfies \eqref{eqn:s_condition}, then with probability at least $1-2\delta$,
%the following bound holds for all $t\geq 0$:
%\begin{multline}\label{eqn:converge_main}
%\norm{u_t - u_\star}^2_2 \leq 2r^t +\Delta^2, 
%\text{\quad where }r = \frac{1-\frac{\alpha}{L^2\beta}}{1 - \sqrt{s_\star/s} - \delta} <1 ~\text{by assumption \eqref{eqn:s_condition}}, \\
%\text{ and }\Delta  = O(n^{-1/3} \sqrt{s (\log n)(\log n\vee p)}) ~\text{, where the exact expression will be given in \eqref{eqn:Delta}.} 
%\end{multline}
\end{theorem}

This theorem provides information about both computation and statistical convergence in estimating $u_\star$. Define \[\Delta = C\left(\eps_n + \frac{s^{1/2}\log(np/\delta)}{n^{1/3}}\right).\] 
Then, from a computation perspective, the algorithm converges linearly up to the time when it reaches the statistical error $\Delta$. In particular, for any tolerance $\tau > \Delta$, running the algorithm for $t \geq \frac{\log\left(\frac{\tau - \Delta}{\sqrt{2}}\right)}{\log(r)}$ many iterations will lead to $\norm{u_t - u_{\star}}^2_2 \leq \tau$ (with probability at least $1-\delta$). 

Next, we consider the dependence on the sample size $n$. Suppose we assume $\epsilon_n \leq \mathcal{O}(n^{-1/3})$ (as we will see in Proposition~\ref{prop:converge_normal_mix} below, this holds with high probability under a random model for $\X$). Then $\Delta$ scales with sample size $n$ as $n^{-1/3}$ (up to log factors). \citet{balabdaoui2019} also obtained $n^{-1/3}$  consistency of a least square estimator of $u_\star$ under similar assumptions. With  $g_\star$ being monotone increasing and twice continuously differentiable, and some further assumptions on the $u_\star$ and $\X$, \citet{balabdaoui2018} proposed a score based estimator which is $\sqrt{n}$-consistent, i.e., error scales with sample size as $n^{-1/2}$. 

Comparing these results, an open question remains regarding the convergence rate of our algorithm in settings where the true $g_\star$ is smooth --- the $n^{-1/2}$ rate achieved by \citet{balabdaoui2018} with slightly stronger assumptions suggests that perhaps our $n^{-1/3}$ rate can be improved with an additional smoothness assumption. In Section \ref{sec:sim_lb} below, we examine this open question empirically, and will see that while the $n^{-1/3}$ rate appears to be tight for a non-smooth $g_\star$, simulation results with a smooth $g_\star$ suggest that an improved rate might be possible in that setting.

\subsubsection{Convergence guarantee for the estimation of $g_\star$}

We show that with the SOD-SIM algorithm in \ref{sec:Alg}, the estimation of $g_\star$ has convergence rate of $\mathcal{O}(\epsilon_n + \frac{s^{1/2\log(np)}}{n^{1/3}})$. In particular, assuming $\epsilon_n \leq O(n^{-1/3})$, the estimation of $g_\star$ has $\mathcal{O}(n^{-1/3})$ convergence rate omitting the log terms.

\begin{proposition}\label{thm:Pred_err}
Assume the conditions~\eqref{eqn:model},~\eqref{eqn:lipschitz},~\eqref{eqn:sparsity}, \eqref{eqn:sparse_eig}, and~\eqref{eqn:ID} hold for the signal,
and condition~\eqref{eqn:noise_assump} holds for the noise. For $\delta>0$, assume \eqref{eqn:eta_s} and \eqref{eqn:s_star_growth} hold for the step size $\eta$, sparsity levels $s$ and $s_\star$, and $n$, $p$. Then as $n$ is sufficiently large, with probability at least $1-\delta$, for all $t \geq 1$, the prediction error is bounded by:
\begin{multline}\label{eqn:Pred_err}
n^{-1/2} \norm{\iso_{\X u_t}\Y - g_\star(\X u_\star)}_2 \leq  L\sqrt{2\beta}\cdot r^t +C'\left(\eps_n + \frac{s^{1/2}\log(np/\delta)}{n^{1/3}}\right),\end{multline}
where $r$ is defined in Theorem \ref{thm:converge_main}, and where the constant $C'$ (specified in the proof)
depends on $\alpha,\beta,L,B,\sigma,\eta,c_s$, but not on $n,p,s,\delta$ and $\epsilon_n$.
\end{proposition}

The proof for Proposition \ref{thm:Pred_err} is deferred to section \ref{pf:Pred_err}. We show our estimator $\iso_{\X u_t}\Y$ of $g_\star$ converges at a $n^{-1/3}$ rate in $\ell_2$ norm. The upper bound of the convergence rate for estimating $g_\star$ matches its lower bound up to the log factors, as we know the minimax rate for isotonic regression is $n^{-1/3}$ \cite{chatterjee2015}. This result matches the convergence results shown in \cite{balabdaoui2019} and \cite{balabdaoui2018}.

\subsection{Convergence guarantees with random design $\X$}\label{sec:conv_randX}
In this section we verify that the  assumptions in section \ref{sec:assumptions} are likely to hold for random design matrices $\X$. In particular, we show the convergence rate for random design $\X$ with normal mixture distribution, which is a much milder condition than the symmetric elliptical $\X$.

\subsubsection{Assumptions}\label{sec:assumption_rand}
We assume the same model \eqref{eqn:model} as in section \ref{sec:assumptions}. We have the following assumptions on $g_\star(\cdot)$, $\X$, $u_\star$ and $\Z$ under random design matrix $\X$.
\paragraph{Assumptions on the signal}
Assume $g_\star(\cdot)$ is bounded Lipschitz and monotone nondecreasing \eqref{eqn:lipschitz}, and $u_\star$ is a sparse unit vector \eqref{eqn:sparsity}. The rows of $\X$ are \iid~draws from a distribution on $X\in\R^p$, and $g_\star(\cdot)$, $X$ and $u_\star$ satisfy that \begin{multline}\label{eqn:randomX_normal} X\sim \sum_{k=1}^K a_k\cdot \mathcal{N}(\mu_k,\Sigma_k),
\text{~i.e.~a mixture-of-normals distribution,}\\ \text{ with $\sum_{k=1}^K a_k =1$, $\norm{\mu_k}_2\leq V$ and $c_0\ident_p\preceq \Sigma_k\preceq c_1\ident_p$ for all $k=1,\dots,K$.} \end{multline} and that \begin{equation}\label{eqn:randomX_var} \VV{g_\star(X^\top u_\star)}\geq \nu^2>0,\end{equation}

\paragraph{Assumptions  on  the  noise} Conditional on $\X$, the noise $\Z$ is subgaussian with scale $\sigma$,
\begin{multline}\label{eqn:subg_randomX}\text{ i.e.~$\EEst{e^{\inner{v}{\Z}}}{\X} \leq e^{\sigma^2\norm{v}^2_2/2}$ holds~}\text{almost surely over $\X$, for any fixed $v\in\R^n$.}\end{multline}

\subsubsection{Convergence guarantees}

\begin{proposition}\label{prop:converge_normal_mix}
Under assumptions \eqref{eqn:model},  \eqref{eqn:lipschitz}, \eqref{eqn:sparsity},  \eqref{eqn:randomX_normal}, \eqref{eqn:randomX_var},
and~\eqref{eqn:subg_randomX}. For $\delta>0$, %assume \eqref{eqn:eta_s} and \eqref{eqn:s_star_growth} hold for the step size $\eta$, sparsity levels $s$ and $s_\star$, and $n$, $p$; 
there exists $\epsilon_n \geq C_\alpha \cdot\sqrt{\frac{s\log(np/\delta)}{n}}$ and constants $C_\alpha$, $\alpha$, $\beta$ depending on $L,c_0,c_1,B,\nu,V$ but not on $n,p,s,\delta$; as n is sufficiently large, with probability at least $1-\delta$, conditions \eqref{eqn:sparse_eig} and \eqref{eqn:ID} hold.

%In addition, with $t$ large enough, the mean integrated squared error satisfy
 %\[ \int (\iso_{\X u_t}Y - g_\star(\X u_\star))^2 dP_X  \leq O_p(n^{-1/3} \sqrt{s (\log n)(\log n\vee p)}).\]
\end{proposition}

Proposition \ref{prop:converge_normal_mix} verifies that under the random $X$ setting proposed in section \ref{sec:assumption_rand}, conditions \eqref{eqn:sparse_eig} and \eqref{eqn:ID} are satisfied with high probability. This implies that when $\X$ has a normal mixture distribution, with high probability, the results of Theorem \ref{thm:converge_main} still apply, that our proposed SOD-SIM algorithm leads to $n^{-1/3}$ convergence rate in $\ell_2$ norm for the estimation of the index $u_\star$. The proof is deferred in section \ref{pf:rand_X}.

\subsection{Proof of Theorem~\ref{thm:converge_main}}\label{sec:pf_conv-fix}

We will now prove our convergence theorem. The detailed proofs for the lemmas are presented later in Section \ref{add_pf}. %For simplicity, we use $\eps$ to denote $\eps_n$ in the proofs.

First, we give two deterministic results, that will use our assumptions~\eqref{eqn:model},~\eqref{eqn:lipschitz},~\eqref{eqn:sparsity}, \ \eqref{eqn:sparse_eig}, and~\eqref{eqn:ID} on the signal (i.e., on $\X$, $g_\star$, and $u_\star$), but hold for any {\em fixed} noise vector $\Z$. Define 
\[\textnormal{Err}_\infty(\Z) = n^{-1}\norm{\X^\top(\Z - \bar{Z}\ones_n)}_{\infty},\]
where $\bar{Z} = \frac{1}{n}\sum_{i=1}^n Z_i$,
and
\[\textnormal{Err}_{\textnormal{iso}}(\Z) = \sup_{u\in\mathbb{S}^{p-1}_s} \left\{n^{-1/2}\norm{\iso_{\X u}(g_\star(\X u_\star )+\Z) - \iso_{\X u}(g_\star(\X u_\star))}_2\right\},\]
where 
\[\mathbb{S}^{p-1}_s = \left\{u\in\R^p : \norm{u}_2=1,\ \textnormal{$u$ is $s$-sparse}\right\}\]
is the set of $s$-sparse unit vectors in $\R^p$.

First, we verify that the initialization $u_0$ defined in~\eqref{eqn:u_0} is well-defined (i.e., $\X^\top(\Y-\bar{Y}\ones_n)\neq 0$), and is not worse than a random guess, meaning that we have $\inner{u_0}{u_\star}\geq 0$.
\begin{lemma}\label{lem:init} 
Assume the conditions~\eqref{eqn:model},~\eqref{eqn:lipschitz},~\eqref{eqn:sparsity},~\eqref{eqn:sparse_eig}, and~\eqref{eqn:ID} hold.
If $s\geq s_\star \cdot \frac{\beta}{\alpha}$
and
\[ \textnormal{Err}_{\infty}(\Z) \leq \frac{\alpha L}{\sqrt{s_\star}},\]
then the initialization $u_0$ defined in~\eqref{eqn:u_0} is well-defined, and
satisfies
\[\inner{u_0}{u_\star}\geq 0.\]
\end{lemma}

Next, we prove a bound on the iterative update step of the algorithm.
\begin{lemma}\label{lem:iter_update_deterministic}
Assume the conditions~\eqref{eqn:model},~\eqref{eqn:lipschitz},~\eqref{eqn:sparsity},~\eqref{eqn:sparse_eig}, and~\eqref{eqn:ID} hold. Fix any $u\in\mathbb{S}^{p-1}_s$ satisfying $\inner{u}{u_\star}\geq 0$, and fix any step size $\eta \in[0, \frac{1}{L\beta}]$ and any sparsity level $s>s_\star$. 
Define a hard-thresholded update step as
\[\check{u} = \frac{\HT(\tilde{u})}{\norm{\HT(\tilde{u})}_2}\textnormal{ where }
\tilde{u} = u +\eta\cdot  \prp{u}\left(\tfrac{1}{n}\X^\top\big(\Y - \iso_{\X u}(\Y)\big)\right).\]
Then
\begin{equation}\label{eqn:iter_update_deterministic}
 \norm{\check{u} - u_\star}_2\leq \left(\frac{1-\alpha\eta L}{1-\sqrt{s_\star/s}}\right)^{1/2} \cdot \norm{u-u_\star}_2  + \textnormal{Remainder}(\Z),\end{equation}
 where
\[\textnormal{Remainder}(\Z) = \eps_n \left(\frac{1}{1-\sqrt{s_\star/s}}\right)^{1/2}+  \frac{2\eta\left(  \sqrt{2s+s_\star} \cdot \textnormal{Err}_\infty(\Z) + \sqrt{\beta} \cdot  \textnormal{Err}_\iso(\Z)\right)}{ 1 - \sqrt{s_\star/s}}.\]
\end{lemma}

Combining these two lemmas proves the following deterministic convergence result:
\begin{lemma}\label{lem:converge_deterministic}
Assume the conditions~\eqref{eqn:model},~\eqref{eqn:lipschitz},~\eqref{eqn:sparsity},~\eqref{eqn:sparse_eig}, and~\eqref{eqn:ID} hold. Fix a step size $\eta \geq0$ and a sparsity level $s\geq 1$ satisfying
\[\eta \leq \frac{1}{L\beta}, \quad s > s_\star \cdot \left(\frac{1}{\alpha\eta L}\right)^2.\]
Assume also that
\[ \textnormal{Err}_{\infty}(\Z) \leq \frac{\alpha L}{\sqrt{s_\star}}.\]
Then for all $t\geq 0$, it holds that
\[\norm{u_t - u_\star}_2\leq \sqrt{2}\cdot r^t + \frac{\textnormal{Remainder}(\Z)}{1-r},\]
where
\[r = \left(\frac{1-\alpha\eta L}{1-\sqrt{s_\star/s}}\right)^{1/2}\]
and where $\textnormal{Remainder}(\Z)$ is defined as in Lemma~\ref{lem:iter_update_deterministic}.
\end{lemma}
\begin{proof}[Proof of Lemma~\ref{lem:converge_deterministic}]
We will prove the lemma by induction.
We will show that, for each $t\geq 0$, it holds that
\begin{equation}\label{eqn:converge_deterministic_induction}
\norm{u_t - u_\star}_2 \leq \sqrt{2}\cdot r^t + \textnormal{Remainder}(\Z) \cdot \sum_{s=1}^t r^{s-1}.\end{equation}
At $t=0$, we have
\[\norm{u_0 - u_\star}^2_2 = \norm{u_0}^2_2 +\norm{u_\star}^2_2 - 2\inner{u_0}{u_\star}\leq 2,\]
since by Lemma~\ref{lem:init} we know that $u_0,u_\star$ are unit vectors satisfying $\inner{u_0}{u_\star}\geq 0$. Therefore~\eqref{eqn:converge_deterministic_induction} holds at $t=0$. Now suppose~\eqref{eqn:converge_deterministic_induction} holds at $t=T\geq 0$. We then have
\begin{multline*}\norm{u_{T+1}-u_\star}_2 \leq r\cdot \norm{u_T-u_\star}_2 + \textnormal{Remainder}(\Z)\\ 
\leq r\cdot \left(\sqrt{2} \cdot r^T + \textnormal{Remainder}(\Z) \cdot \sum_{s=1}^T r^{s-1}\right)+ \textnormal{Remainder}(\Z),\end{multline*}
where the first step holds by Lemma~\ref{lem:iter_update_deterministic} and the second step applies~\eqref{eqn:converge_deterministic_induction} with $t=T$. This proves that~\eqref{eqn:converge_deterministic_induction} holds with $t=T+1$, completing the proof.
\end{proof}

With these deterministic results in place, we now need to bound $\textnormal{Err}_{\infty}(\Z)$ and $\textnormal{Err}_{\iso}(\Z)$, under our assumptions on the noise $\Z$.
\begin{lemma}\label{lem:bound_err_terms}
Assume the conditions~\eqref{eqn:model},~\eqref{eqn:lipschitz},~\eqref{eqn:sparsity},~\eqref{eqn:sparse_eig}, and~\eqref{eqn:ID} on the signal, and~\eqref{eqn:noise_assump} on the noise. For any $\delta>0$, with probability at least $1-\delta$ it holds that
\[\textnormal{Err}_\infty(\Z)\leq\sigma\sqrt{\frac{2\beta \log(4p/\delta)}{n}}\]
and
\[\textnormal{Err}_\iso(\Z)\leq n^{-1/3}\left(2B + \sigma\sqrt{8\log n}\cdot \sqrt{\log\left(\frac{3n^{2s+1}p^s}{\delta}\right)}\right).\]
%=: \textnormal{Bound}_2(\delta).\]
\end{lemma}

Combining Lemma~\ref{lem:converge_deterministic} with Lemma~\ref{lem:bound_err_terms}, we have proved Theorem~\ref{thm:converge_main} with the constant
\[C = \max\left\{\frac{r/(1-r)}{\sqrt{1-\alpha \eta L}}, \frac{r^2/(1-r)}{1-\alpha \eta L}\cdot \eta\sqrt{\beta}\left(7 +4B+ 13\sigma\right)\right\}.\]

\section{Experiments}

In this section, we first use simulation experiments to demonstrate the performance of the SOD-SIM algorithm in estimating $u_\star$, and explore the lower bound for the convergence rate; then we apply the SOD-SIM algorithm to a classification problem with PU rocker protein data.

\subsection{Simulation studies}

\subsubsection{Performance in estimating $u_\star$ for PU data}

In this section, we study the performance of the SOD-SIM algorithm in estimating $u_\star$ for PU data without knowing the prevalence $\pi$. 

\paragraph{Methods} We simulate PU data with $400$ positive data and $400$ unlabeled data. Samples in the population are independent and for individual $i$, $X_i \in \R^{p} \sim \normal (0, \Sigma_{\rho})$, where $\Sigma_\rho$ is the autoregressive matrix with its $i,j-$th entry being $\rho^{|i-j|}$, where we let $\rho=0.2$. Next, let $Y_i | X_i \sim \text{Bernoulli}\left(g(X_i^\top u_\star)\right)$, where $g(t) = \frac{e^t}{1+e^t}$ is the expit function, $u_\star = (\frac{\sqrt{2}}{2}, -\frac{\sqrt{2}}{2},0,\cdots,0) \in \R^p$ is with sparsity level $s_\star=2$ and we vary $p = 100, 400, 800, 1600$.

We compare the performance of our proposed SOD-SIM in estimating $u_\star$ with the logistic model with $\ell_1$ penalized maximum likelihood (sparseLR) method. For the sparseLR, we choose the tuning parameter using cross validation. For the SOD-SIM algorithm, we let the working sparsity $s=10$ and set the learning rate $\eta=0.1$. We simulate for $M=100$ times.

\paragraph{Results} The performance is shown in Figure \ref{fig:simulation}. Since $u_\star$ is only identifiable up to direction, both $u_\star$ and the estimations $\widehat{u}$ are rescaled to have unit norms. Besides bias, standard deviation (SD) and rooted mean square error (RMSE), we also characterize the mean inner product of $u_\star$ and $\widehat{u}$ as their correlation. We can see that for all settings with different $p$, the SOD-SIM method has smaller bias and SD than the sparseLR method and the correlation from SOD-SIM with $u_\star$ is closer to $1$. With the dimension of the covariates $p$ increases, the performance of the sparseLR method becomes worse, whereas the SOD-SIM method remains a good performance. The results shows that the mis-specification of the link function and the non-symmetric covariate distribution of the PU data affects the estimation of $u_\star$ using the parametric sparseLR method. The proposed SOD-SIM method performs well for the PU data without specifying the prevalence $\pi$.  

\begin{figure}
    \centering
    \includegraphics[scale=0.9]{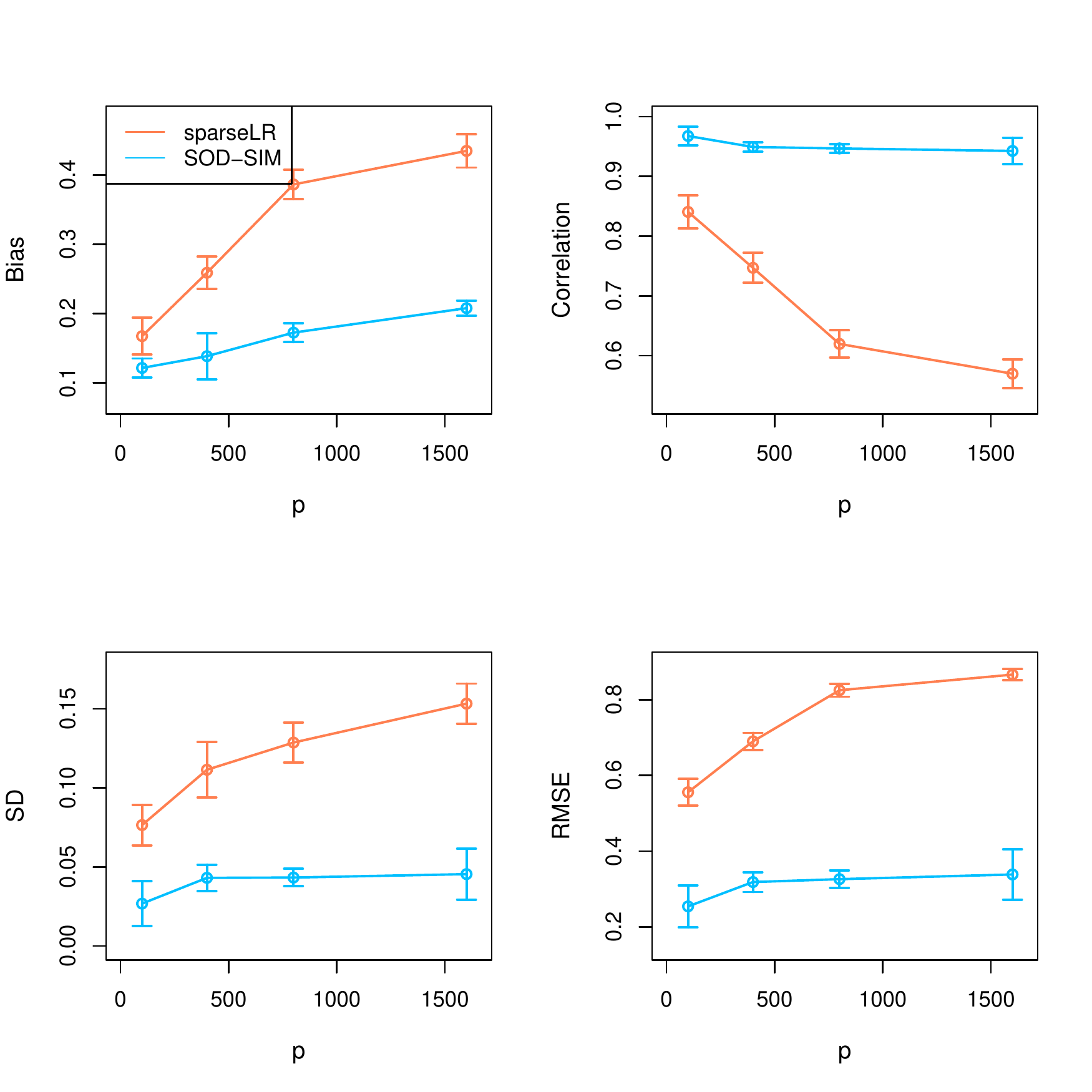}
    \caption{Empirical performance of the SOD-SIM and sparseLR algorithms on PU data in terms of bias, SD, correlation and RMSE with their $95\%$ confidence intervals.}
    \label{fig:simulation}
\end{figure}
 
 \subsubsection{Lower bound exploration}\label{sec:sim_lb}
In this section we study whether our convergence rate for $\norm{u_t - u_\star}_2$ is tight in terms of its dependence on the sample size $n$. We are interested in two questions: (1) Is the upper bound from Theorem \ref{thm:converge_main} tight for the SOD-SIM algorithm? (2) Does the SOD-SIM algorithm attain the optimal convergence rate of the global minimizer? 

%In particular, \citet{balabdaoui2018} has shown that assuming $g_\star$ is twice continuously differentiable with bounded derivatives, and $\X$ distribution satisfies some moment conditions, a score based estimator can achieve $n^{-1/2}$ convergence rate. We would like to understand whether the $n^{-1/3}$ rate is optimal under our milder assumptions. 
\paragraph{Methods}
We vary the sample size $n = 100, 150, 200, 300, 400, 600, 800, 1200$; and construct two low-dimensional examples with $s_\star=2$, $p=2$, $u^\star = (\frac{\sqrt{2}}{2}, \frac{\sqrt{2}}{2})$. We let $\Y = g_n(\mathbf{X}u^\star) + \Z$ where $\Z \sim \normal(0,\sigma^2 I)$. In Example 1, $\sigma =1$ and in Example 2, $\sigma=0.8$. 

We first construct the $g_n(x)$ function. In Example 1, we construct $g_n(x)$ of which the second derivative is unbounded as $n \rightarrow \infty$. Specifically, we have 
\begin{equation}\label{eqn:gn1} g_n(x) = x - \frac{n^{-1/3}f(n^{1/3}x)}{2+\epsilon}~ \text{where}~ f(x) = 2(x-\lfloor x\rfloor) - 4(x-\lfloor x\rfloor -1/2)_+,\end{equation} where we set $\epsilon = 0.1$. Notice that in this construction, $g_n''(x) \propto  n^{1/3} f''(n^{1/3}x) \propto n^{1/3}$.

In Example 2, we construct a smoother version of $g_n(x)$ by letting
 \begin{equation}\label{eqn:gn3} g_n(x) = x - \frac{n^{-2/3}f(n^{1/3}x)}{2+\epsilon}\end{equation} where we set $\epsilon = 0.1$. For this example, $g_n(x)$ has bounded limiting second derivative as $n\rightarrow \infty$. These two example functions are plotted in Figure \ref{fig:sim}.  

To construct the covariates, for each sample, we first generate $t_i \sim \textnormal{Unif}[0,1]$ and we define $X_i$ as 
\begin{align*} & X_{1i} =\frac{\sqrt{2}}{\textnormal{sd}_1} n^{-1/3}f(n^{1/3} t_i),\\& X_{2i} =\frac{\sqrt{2}}{\text{sd}_2} ( t_i - n^{-1/3}f(n^{1/3} t_i)) + \sqrt{2} \varepsilon_{i} ,\end{align*} 
for $i=1,\cdots,n$, where $\varepsilon_{i} \sim \text{Unif}[-0.5,0.5]$ are i.i.d. and independent of $X_{1i}$. The constants $\text{sd}_1$ and $\text{sd}_2$ are used to make $X_1$ and $X_2$ to have roughly equal variances. In particular, in Example 1, we let $\text{sd}_1=0.025$ and $\text{sd}_2=0.4$, and in Example 2, we let $\text{sd}_1=0.0014$ and $\text{sd}_2=0.4$. 
%We then sample from these $n$ samples generated with replacement.\\

%We compute the MSE for $\hat{u}$ using $M=1000$ simulations. For all settings we choose learning rate $\eta = 0.2$ and run the algorithm for 150 iterations. We choose the learning rate based on the convergence of $\ell_2$ error at convergence. We plotted the rate of convergence.\\ 

We compute the global minimizer of the $\ell_2$ loss. For $\theta \in [0, \pi/2]$, with increment $0.001$, we compute $\norm{\Y - \textnormal{iso}_{\mathbf{X}u_{\theta}}(\Y)}_2$, where $u_{\theta} = (\cos \theta, \sin \theta)$, to obtain the minimizer $\hat{u}_g$ of the $\ell_2$ loss.

\paragraph{Results}
\begin{figure}
\begin{center}
\includegraphics[scale=0.9]{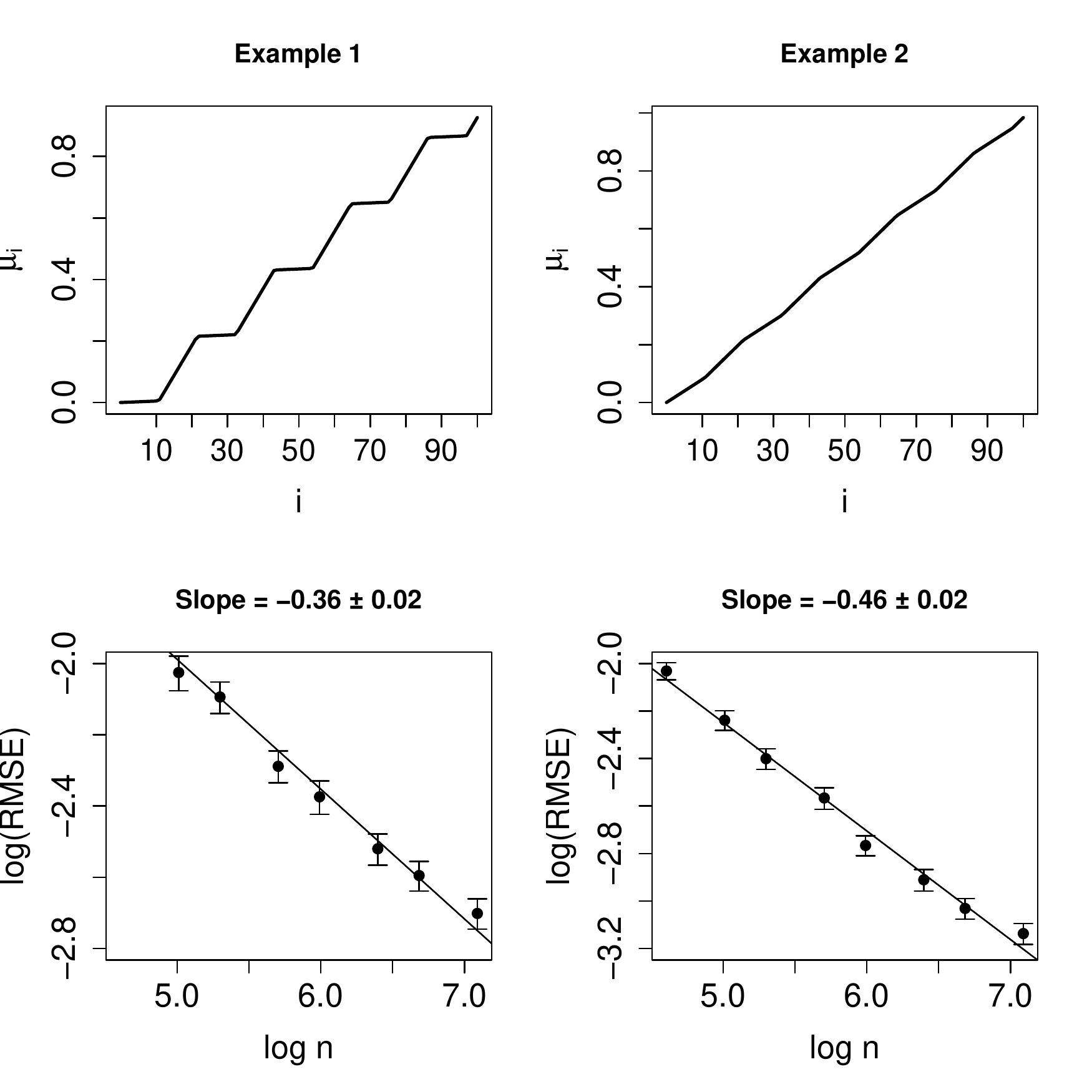}
\end{center}
\caption{Upper: $g_n$ from \eqref{eqn:gn1} (left) and \eqref{eqn:gn3} (right). Lower: $\log(\text{RMSE})$ (with 95\% confidence interval) vs $\log(n)$ for the global estimator for Example 1 \eqref{eqn:gn1} (left) and Example 2 \eqref{eqn:gn3} (right).}
\label{fig:sim}
\end{figure}

We plot the logarithm of rooted mean square error ($\log(\text{RMSE})$) against $\log(n)$ to show the rate of convergence. Example 1 (non-smooth example) has a convergence rate close to $n^{-1/3}$, which indicates our proved convergence rate for the SOD-SIM algorithm is as tight as the global minimizer. The second example (with bounded limiting second derivative) has a faster convergence rate, which is close to $n^{-1/2}$. 

Our exploratory simulation results match with the reported convergence rate of the score based estimator in \cite{balabdaoui2018}. We do not have a definitive answer for the convergence rate of our algorithm when the function $g_\star$ is smooth. A better rate could potentially be achieved with alternative proof techniques, and is beyond the scope of this work.
%\begin{table}
%\begin{center}
%\begin{tabular}{c|cc|cc|cc}
%\hline
%& $p=100$ & & $p=400$ & & $p=1600$ & \\
%\hline
%& sparseLR & HOG-SIM & sparseLR & HOG-SIM & sparseLR & HOG-SIM\\
%Bias & 0.13 & 0.05 & 0.17 & 0.08 & 0.28 & 0.17\\
%SD & 0.05& 0.03 & 0.07 & 0.04 & 0.12 & 0.05\\
%RMSE & 0.50 & 0.23 & 0.57 & 0.25 & 0.69 & 0.32\\
%Correlation & 0.88 & 0.97 & 0.83 & 0.97 & 0.73 & 0.95\\
%\hline
%\end{tabular} 
%\end{center}
%\caption{Empirical performance of our algorithm (HOG-SIM) vs sparseLR. {\color{magenta} [[[need to replace p with $p$ / n with $n$ / etc in any table headings, captions, etc RD: done; will redo the figures]]]}\label{table:sim}}
%\end{table}

\subsection{Application to rocker protein data}	
In this section, we study an application of our SOD-SIM algorithm to the rocker protein sequence data. The rocker data is composed of sequences with confirmed function ($Y=1$) and sequences with unknown functionality ($Y=0$). 

To give some biological context to this dataset, Rocker is a \emph{de novo} designed protein which recreates the biological function of substrate transportation \citep{Joh2014-xy}. A functional protein successfully transports ions across cell membranes. To study how a mutation at each site affects the membrane transport ability of proteins, mutations are introduced to the original Rocker sequence (wild-type) by double site saturation mutagenesis. %Mutated sequences differ by at most two sites from the wild-type sequence.  
Mutated sequences are screened by fluorescence-activated cell sorting (FACS), which sort out \emph{functional} protein variants. Due to an experimental challenge of obtaining negative sequences, an additional set of sequences are obtained from the initial library whose associated functionality is unknown. Therefore, the resulting data set is a Positive-Unlabeled (PU) dataset, because the first set consists of functional sequence variants and the second set consists unlabeled examples.

\paragraph{Methods} A sequence consists of 25 positions taking one of the $21$ discrete values, which correspond to 20 amino acid letter codes plus an additional letter for the alignment gap. Each sequence contains at most two mutations. The data consists of $n_\ell = 703030$ functional (positive label) and $n_u = 1287155$ unlabeled sequences. To compare different algorithms reliably, we split the data into 10 subsets, and use half of each split for training and remaining half for testing. Each training dataset contains 35K labeled and 64K unlabeled examples on average.  We generate features with main (site-wise) and pairwise (interaction between two sites) effects using one-hot encoding of the sequences. Removing columns with zero counts, we obtain 27K features on average, where 490 of the generated features correspond to the main effects. We take the amino acid levels in the WT sequence as the baseline levels to generate a sparse design matrix $\mathbf{X} \in \mathbb{R}^{n\times p}$ where $(n,p) \approx$ (100K, 27K). The number of unique sequences is around 26K (i.e., rank of row space of $\mathbf{X}\approx$ 26K), which makes the problem high-dimensional. The response vector $\mathbf{Y}\in \mathbb{R}^{n}$ represents whether each sequence $i$ is labeled ($\mathbf{Y}_i =1$) or unlabeled ($\mathbf{Y}_i =0$). %Finally, grouping sequences into their unique levels, we obtained a sparse design matrix $\mathbf{X} \in \mathbb{R}^{n \times p}$, a response $\mathbf{Y} \in \mathbb{R}^{n\times 1}$, and a diagonal weight matrix $\mathbf{W} \in \mathbb{R}^{n\times n}$, where each $\mathbf{Y}_i \in \{0,1\}$ represents whether the $i$th sequence is labeled ($\mathbf{Y}_i =1$) or unlabeled ($\mathbf{Y}_i =0$), and a diagonal element of the weight matrix $\mathbf{W}_{ii}$ corresponds to a count of the 

We applied four methods to each train dataset which we denote as follows:
\begin{itemize}
	\item SOD-SIM: our proposed algorithm
	\item sparseLR: the logistic regression with $\ell_1$-penalty
	\item PV1: the proposed method in \cite{Plan2013-ew}. We solve the following objective:
$$ \hat{u}^{PV1}  \in \arg \max_{u \in K(s)} \langle y, \mathbf{X} u\rangle$$
where $K(s)$ is a $s$-approximate sparse set, defined as $K(s) := \{x\in \mathbb{R} : \|x\|_2 \leq 1,  \|x\|_1\leq \sqrt{s}\}$
    
	\item PV2: the proposed method in \cite{Plan2016-cz}. 
$$ \hat{u}^{PV2}  \in \arg \min_{u \in K(s)} \|y-\mathbf{X} u \|_2. $$
\end{itemize}

We used the \textbf{glmnet} package to solve the logistic regression objective with $\ell_1$-penalty. To solve objectives in PV1 and PV2, we implemented the projected gradient methods, where we iteratively projected the gradients onto the intersection of $\ell_1$ and $\ell_2$ balls. We used the Dykstra's projection algorithm to obtain the projection onto the intersection of the two balls \citep{Boyle1986-us}. Note both objectives in PV1 and PV2 are convex, and therefore the projected gradient descent algorithm guarantees to find a global minimum. For PV1 and PV2, we input the standardized $\mathbf{X}$ where each column is centered and scaled. Since sparseLR, PV1, and PV2 are convex problems, we run until the algorithms the models converge. For SOD-SIM, we let the learning rate $\eta=1$ and run the algorithm until changes of estimated parameters is small ($< 0.0005$) or the iteration number $t$ reaches the pre-defined maximum number of iterations ($\leq 1000$). In addition, PV1 and PV2 methods do not provide estimates of the link function. We run an isotonic regression after obtaining $\hat{u}^{PV1}$ and $\hat{u}^{PV2}$ to estimate the link function, and used such estimates to perform downstream prediction tasks. 

We used four metrics (Accuracy, F1 score, Brier Score, and AUC value) to evaluate predictive performance of the four methods. Accuracy is the proportion of correctly classified examples. F1 score is the harmonic mean of the precision and recall, whose value lies between 0 and 1. Higher number corresponds to a better performance. Brier Score is an average squared $\ell_2$ loss, i.e., Brier Score := $\frac{1}{n_{test}} \sum_{i=1}^{n_{test}} (y_i - \hat{y}_i)^2$, and therefore, lower numbers correspond to better performances. AUC measures the area under the ROC curve. The perfect classification corresponds to the AUC value of 1, and a random classification corresponds to 0.5. For the choice of hyperparameters ($s$ for SOD-SIM, PV1, and PV2, and $\lambda$ for the $\ell_1$ logistic regression), we use a grid of 100 working sparsity $s$ values from 1 to $p$, interpolated in a square root scale, and a grid of 100 lambda values obtained from the \textbf{glmnet} package by default. We picked the hyperparameters that gave the best results for the most of the metrics on the test datasets. 

\paragraph{Results} Figure \ref{fig:fig1} demonstrates the empirical performances of the four methods. SOD-SIM performed the best in all metrics on average. The predictive performance of the sparse logistic regression was slightly worse than the SOD-SIM. It is likely due to the mis-specification error, since we are forcing the link function to be a sigmoid function. Both PV1 and PV2 seem to suffer from the deviation $\mathbf{X}$ from the Gaussian design; however, PV2 seems to be more robust to such deviation.

\begin{figure}[ht]
\centering
\includegraphics[scale=0.09]{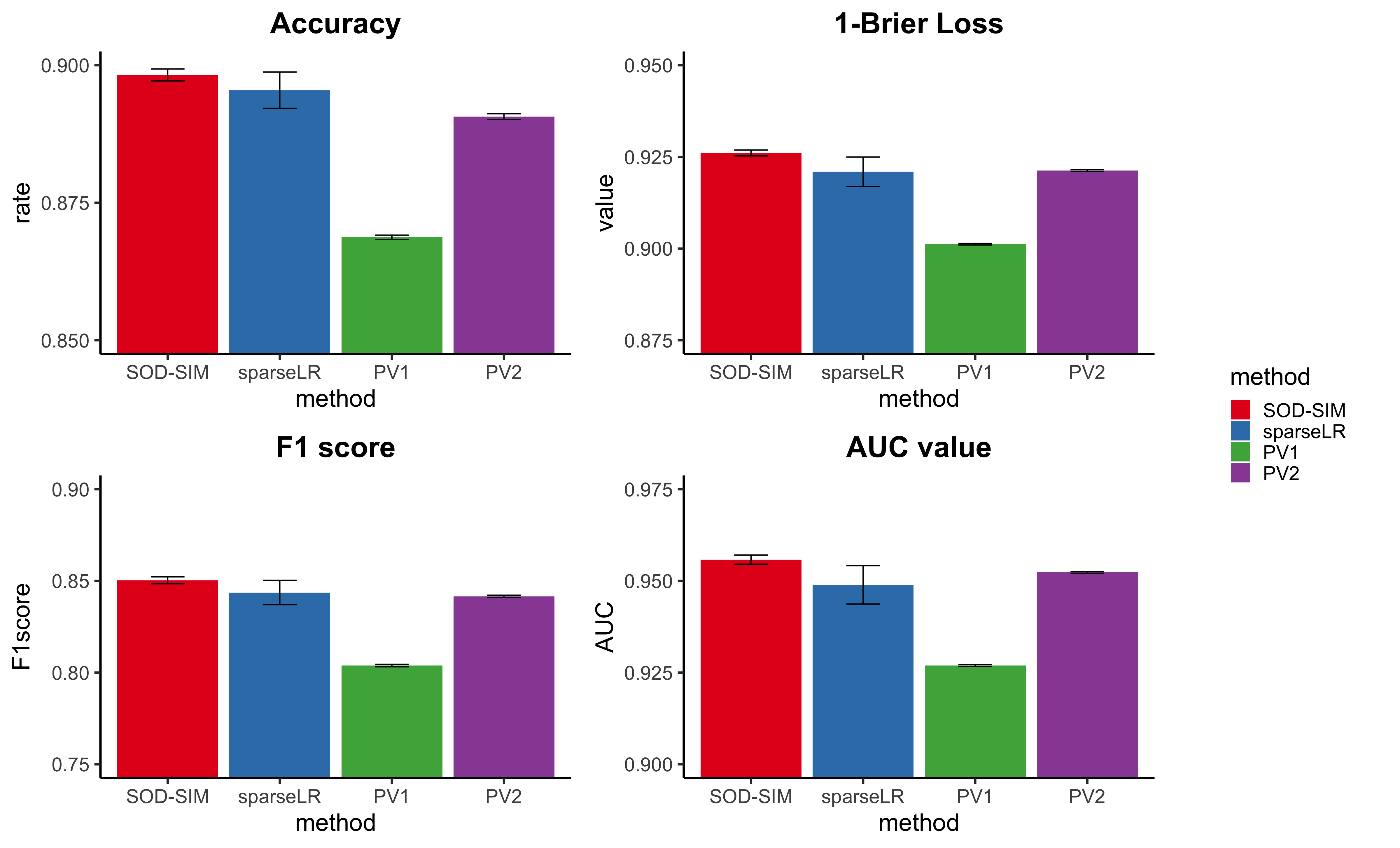}	
\caption{Average accuracy rates, F1 scores, 1-Brier Score, and AUC values of the four methods on 10 test datasets. Higher values correspond to better performance in all four plots. The error bars represent one standard error.}
\end{figure}\label{fig:fig1}

\section{Discussion}

In this paper, we propose a scalable iterative projection-based algorithm (SOD-SIM) for the estimation of hdSIMs, which we show the statistical convergence guarantees for the estimation of both the coefficient $u_\star$ and the unknown function $g_\star$. From the minimax results of isotonic regression perspective, our convergence rate for $g_\star$ is tight with respect to $n$ up to some poly log terms. Our simulation results suggest that under the mild model assumptions, the convergence rate for the estimation of $u_\star$ is also tight. However, our simulation results for the smooth $g_\star$ case suggests that the convergence rate might potentially decay faster with sample size $n$, under stronger model assumptions---while the theoretical upper bounds proved here scale as $n^{-1/3}$ (ignoring log terms), the simulations suggest that the lower bound for the settings with additional smoothness conditions might potentially be improved. The gap between these two rates in the smooth case remains an open question. 

Our statistical guarantee requires very mild model assumptions, especially for the covariates $\X$. This allows our theoretical guarantee to cover many useful examples such as the PU data. From our simulation studies, we have shown that even if at population level the covariates $\X$ is normally distributed, the positive-only data has asymmetric distribution and parametric models such as sparseLR suffer from estimation bias caused by both misspecification of the link function and the asymmetric distribution of the covariates; whereas our SOD-SIM algorithm has good performance in estimating $u_\star$ in high-dimensional settings. In our real data analysis, we have shown that our method (SOD-SIM) also outperforms methods based on the slicing regression idea (PV1, PV2).

\subsection*{Acknowledgements}
R.F.B. was partially supported by the National Science Foundation via grants DMS-1654076 and DMS-2023109, and by the Office of Naval Research via grant  N00014-20-1-2337. G.R.~was partially supported by the National Science Foundation via grant DMS-1811767 and by the National Institute of Health via grant R01 GM131381-01. H. S.~was partially supported by the National Institute of Health via grant R01 GM131381-01. The authors thank Sabyasachi Chatterjee for helpful discussions.

% Format a LaTeX bibliography
%\bibliographystyle{ieeestyle}
\bibliographystyle{abbrvnat}
\bibliography{thesis}
% Figures and tables, if you decide to leave them to the end
%\input{figure}
%\input{table}
\appendix
\section{Additional proofs}\label{add_pf}
%For notation simplicity, we use $\eps$ to denote $\eps_n$ in the proofs.
\subsection{Proof of Lemma~\ref{lem:iter_update_deterministic}}
First, writing $\bar{u} = \HT(\tilde{u})$,
since $u$ is a $s$-sparse unit vector we can calculate
\[1 = \norm{u}^2_2 = \inner{\tilde{u}}{u} \leq  \norm{\tilde{u}_{\textnormal{support}(u)}}_2 \leq \max_{S\subset[d],|S|=s} \norm{\tilde{u}_S}_2 = \norm{\bar{u}}_2 ,\]
where the second equality holds
since $\tilde{u}$ is equal to $u$ plus an orthogonal vector. Since $u_\star$ is a unit vector, we have
\[\norm{\check{u} - u_\star}_2 \leq \norm{\bar{u} - u_\star}_2,\]
since $\check{u} = \bar{u}/\norm{\bar{u}}_2$ is the projection of $\bar{u}$ onto the unit ball, which is a convex set containing $u_\star$. 

A trivial calculation shows that
\begin{equation}\label{eqn:step1_deterministic}
\norm{\bar{u} - u_\star}^2_2=\norm{u - u_\star}^2_2 - \norm{u - \bar{u}}^2_2 + 2\inner{u - \bar{u}}{u_\star - \bar{u}}.\end{equation}
Now we work with this inner product. By definition of $\tilde{u}$, we can write
\begin{multline*}
\inner{u -\bar{u}}{u_\star - \bar{u}} = \inner{\tilde{u} - \bar{u}}{u_\star - \bar{u}} + \frac{\eta}{n} \inner{\prp{u}\X^\top \big(\Y - \iso_{\X u}(\Y)\big)}{\bar{u} - u_\star}\\
\leq \frac{\sqrt{s_\star}}{2\sqrt{s}}\norm{u_\star - \bar{u}}^2_2+ \frac{\eta}{n} \inner{\prp{u}\X^\top \big(\Y - \iso_{\X u}(\Y)\big)}{\bar{u} - u_\star},
\end{multline*}
where the last step applies Lemma~\ref{lem:HT} (presented later on). Plugging this back into~\eqref{eqn:step1_deterministic} and rearranging terms, we obtain
\[
\left(1 - \sqrt{\frac{s_\star}{s}}\right) \norm{\bar{u} - u_\star}^2_2 \leq \norm{u - u_\star}^2_2 - \norm{u - \bar{u}}^2_2 +  \frac{2\eta}{ n}\inner{\prp{u}\X^\top \big(\Y - \iso_{\X u}(\Y)\big)}{\bar{u} - u_\star}.\]
Moreover, noting that $\ones_n^\top(y-\iso_v(y))=0$ for any $y,v\in\R^n$ by properties of isotonic regression, we can equivalently write this as
\begin{multline}\label{eqn:step2_deterministic}
\left(1 - \sqrt{\frac{s_\star}{s}}\right) \norm{\bar{u} - u_\star}^2_2 \\\leq \norm{u - u_\star}^2_2 - \norm{u - \bar{u}}^2_2 +  \frac{2\eta}{ n}\inner{\prp{u}\X^\top\prp{\ones_n} \big(\Y - \iso_{\X u}(\Y)\big)}{\bar{u} - u_\star}.\end{multline}
Next, working with this last term, denote $\mu_\star = g_\star(\X u_\star)$, we can write
\begin{multline*}\prp{\ones_n} \big(\Y - \iso_{\X u}(\Y)\big)\\
=  \prp{\ones_n}\left(\mu_\star - \iso_{\X u}(\mu_\star)\right) +\prp{\ones_n}\left(\Y - \mu_\star\right) + \prp{\ones_n}\left(\iso_{\X u}(\mu_\star) - \iso_{\X u}(\Y) \right)\\
=  \left(\mu_\star - \iso_{\X u}(\mu_\star)\right) +\left(\Z - \bar{Z}\ones_n\right) + \prp{\ones_n}\left(\iso_{\X u}(\mu_\star) - \iso_{\X u}(\Y) \right),\end{multline*}
and since
\begin{multline*}\inner{\prp{u}\X^\top \prp{\ones_n} \big(\Y - \iso_{\X u}(\Y)\big)}{\bar{u} - u_\star}
\leq \inner{\prp{u}\X^\top \big(\mu_\star - \iso_{\X u}(\mu_\star)\big)}{\bar{u} - u_\star} \\{}+ \Norm{\X^\top \left(\Z - \bar{Z}\ones_n\right)}_{\infty} \norm{\prp{u}\big(\bar{u}-u_\star\big)}_1 + 
 \norm{ \iso_{\X u}(\Y) - \iso_{\X u}(\mu_\star)}_2\norm{\prp{\ones_n}\X\prp{u}\big(\bar{u}-u_\star\big)}_2,\end{multline*}
 and therefore plugging in our definitions of $\textnormal{Err}_\infty(\Z)$ and  $\textnormal{Err}_\iso(\Z)$ from above,
\begin{multline*}\inner{\prp{u}\X^\top \big(\Y - \iso_{\X u}(\Y)\big)}{\bar{u} - u_\star}
\leq \inner{\prp{u}\X^\top \big(\mu_\star -\iso_{\X u}(\mu_\star)\big)}{\bar{u} - u_\star} \\
+n\textnormal{Err}_\infty(\Z)\cdot \norm{\prp{u}\big(\bar{u}-u_\star\big)}_1 + 
 n^{1/2}\textnormal{Err}_\iso(\Z)\cdot \norm{\prp{\ones_n}\X\prp{u}\big(\bar{u}-u_\star\big)}_2.\end{multline*}

Since $u,\bar{u}$ are $s$-sparse and $u_\star$ is $s_\star$-sparse, we can calculate that $\prp{u}\big(\bar{u}-u_\star\big)$ is $(2s+s_\star)$-sparse and therefore
\[\norm{\prp{u}\big(\bar{u}-u_\star\big)}_1 \leq \sqrt{2s+s_\star}\norm{\prp{u}\big(\bar{u}-u_\star\big)}_2\leq \sqrt{2s+s_\star}\norm{\bar{u}-u_\star}_2.\]
By our condition~\eqref{eqn:sparse_eig} bounding the sparse eigenvalues of $\X$, we also have 
\[\norm{\prp{\ones_n}\X\prp{u}\big(\bar{u}-u_\star\big)}_2\leq\norm{\X\prp{u}\big(\bar{u}-u_\star\big)}_2 \leq \sqrt{\beta n}\norm{\prp{u}\big(\bar{u}-u_\star\big)}_2 \leq \sqrt{\beta n}\norm{\bar{u}-u_\star}_2.\]
Combining everything and returning to~\eqref{eqn:step2_deterministic}, then,
\begin{multline}\label{eqn:step3_deterministic}
\left(1 - \sqrt{\frac{s_\star}{s}}\right) \norm{\bar{u} - u_\star}^2_2 \leq \norm{u - u_\star}^2_2 - \norm{u - \bar{u}}^2_2\\
+   2\eta \norm{\bar{u}-u_\star}_2 \left(  \sqrt{2s+s_\star} \cdot \textnormal{Err}_\infty(\Z) + \sqrt{\beta} \cdot  \textnormal{Err}_\iso(\Z)\right)\\
{} +  \frac{2\eta}{ n} \inner{\prp{u}\X^\top \big(\mu_\star - \iso_{\X u}(\mu_\star)\big)}{\bar{u} - u_\star}.
\end{multline}
Next we work with the remaining inner product. We have
\begin{align*}
& \inner{\prp{u}\X^\top \big(\mu_\star - \iso_{\X u}(\mu_\star)\big)}{\bar{u} - u_\star} \\
&=  \inner{\mu_\star - \iso_{\X u}(\mu_\star)}{\X\prp{u} (\bar{u}-u_\star))} \\
&=  \inner{\mu_\star - \iso_{\X u}(\mu_\star)}{\X\prp{u} \bar{u}} -  \inner{\mu_\star - \iso_{\X u}(\mu_\star)}{\X\prp{u} u_\star} \\
&=  \inner{\mu_\star - \iso_{\X u}(\mu_\star)}{\X\prp{u} (\bar{u}-u)} -  \inner{\mu_\star - \iso_{\X u}(\mu_\star)}{\X(u_\star-u \cdot u^\top u_\star)} \\
& \leq \norm{\mu_\star - \iso_{\X u}(\mu_\star)}_2\cdot\norm{\X\prp{u}(\bar{u} - u)}_2-  \inner{\mu_\star - \iso_{\X u}(\mu_\star)}{\X u_\star} \\
&\hspace{1in}+ \inner{u}{u_\star}\cdot \inner{\mu_\star - \iso_{\X u}(\mu_\star)}{\X u}\\
& \leq \norm{\mu_\star - \iso_{\X u}(\mu_\star)}_2\cdot \sqrt{\beta n}\norm{\bar{u} - u}_2-\frac{1}{L}\norm{\mu_\star - \iso_{\X u}(\mu_\star)}^2_2 ,\end{align*}
where to prove the last step, we apply~\eqref{eqn:sparse_eig} to the first term, apply Lemma~\ref{lem:ID} (presented later) to the second term, and for the third term we observe that it is
$\leq 0$, because $\inner{u}{u_\star}\geq 0$
by assumption while $\inner{\mu_\star - \iso_{\X u}(\mu_\star)}{\X u}\leq 0$ by definition of isotonic regression (i.e., since isotonic regression is projection onto a convex cone).
Finally, by Cauchy--Schwarz we have
\[ \norm{\mu_\star - \iso_{\X u}(\mu_\star)}_2\cdot \sqrt{\beta n}\norm{\bar{u} - u}_2 \leq  \frac{L\beta n}{2}\norm{\bar{u} - u}^2_2+ \frac{1}{2L}\norm{\mu_\star - \iso_{\X u}(\mu_\star)}^2_2.\]
Therefore, we have calculated 
\[\inner{\prp{u}\X^\top \big(\mu_\star - \iso_{\X u}(\mu_\star)\big)}{\bar{u} - u_\star} \leq \frac{L\beta n}{2}\norm{\bar{u} - u}^2_2-\frac{1}{2L}\norm{\mu_\star - \iso_{\X u}(\mu_\star)}^2_2.\]
Plugging these calculations back into~\eqref{eqn:step3_deterministic}, 
and applying the assumption that  $\eta\leq \frac{1}{L\beta}$, we obtain
\begin{multline}\label{eqn:step4_deterministic}
\left(1 - \sqrt{\frac{s_\star}{s}}\right) \norm{\bar{u} - u_\star}^2_2 \leq \norm{u - u_\star}^2_2-  \frac{\eta}{Ln} \norm{\mu_\star - \iso_{\X u}(\mu_\star)}^2_2 \\
+   2\eta \norm{\bar{u}-u_\star}_2 \left(  \sqrt{2s+s_\star} \cdot \textnormal{Err}_\infty(\Z) + \sqrt{\beta} \cdot  \textnormal{Err}_\iso(\Z)\right).
\end{multline}

Next, we split into cases. If $\norm{u-u_\star}\leq\eps_n$, then~\eqref{eqn:step4_deterministic} implies that
\[\left(1 - \sqrt{\frac{s_\star}{s}}\right) \norm{\bar{u} - u_\star}^2_2
\leq \eps_n^2 +  2\eta \norm{\bar{u}-u_\star}_2 \left(  \sqrt{2s+s_\star} \cdot \textnormal{Err}_\infty(\Z) + \sqrt{\beta} \cdot  \textnormal{Err}_\iso(\Z)\right)\]
which can be relaxed to
\[\norm{\bar{u} - u_\star}_2
\leq \frac{\eps_n}{\left(1-\sqrt{\frac{s_\star}{s}}\right)^{1/2}} + \frac{2\eta\left(\sqrt{2s+s_\star} \cdot \textnormal{Err}_\infty(\Z) + \sqrt{\beta} \cdot  \textnormal{Err}_\iso(\Z)\right)}{1-\sqrt{\frac{s_\star}{s}}}. \]
If instead $\norm{u-u_\star}_2\geq \eps_n$, then
by the identifiability assumption~\eqref{eqn:ID}, 
noting  that by the definition of isotonic regression, we can write $ \iso_{\X u}(\mu_\star) = g(\X u)$ for some monotone non-decreasing function $g$, we have
\[\frac{1}{L^2n}\norm{\mu_\star - \iso_{\X u}(\mu_\star)}^2_2 = \frac{1}{L^2n}\norm{g_\star(\X u_\star) - g(\X u)}^2_2 \geq \alpha \norm{u-u_\star}^2_2.\]
In this case,~\eqref{eqn:step4_deterministic} implies that
\begin{multline*}\left(1 - \sqrt{\frac{s_\star}{s}}\right) \norm{\bar{u} - u_\star}^2_2
\\\leq \left(1 -\alpha\eta L\right)\norm{u-u_\star}^2_2 +  2\eta \norm{\bar{u}-u_\star}_2 \left(  \sqrt{2s+s_\star} \cdot \textnormal{Err}_\infty(\Z) + \sqrt{\beta} \cdot  \textnormal{Err}_\iso(\Z)\right)\end{multline*}
which can be relaxed to
\[\norm{\bar{u} - u_\star}_2
\leq \left(\frac{1- \alpha\eta L}{1-\sqrt{\frac{s_\star}{s}}}\right)^{1/2}\norm{u-u_\star}_2 + \frac{2\eta\left(\sqrt{2s+s_\star} \cdot \textnormal{Err}_\infty(\Z) + \sqrt{\beta} \cdot  \textnormal{Err}_\iso(\Z)\right)}{1-\sqrt{\frac{s_\star}{s}}}. \]
Finally, since we proved above that $ \norm{\check{u} - u_\star}_2\leq \norm{\bar{u} - u_\star}_2$,
this completes the proof.

\subsubsection{Supporting lemmas}
\begin{lemma}[Liu \& Barber 2018, Lemma 1]\label{lem:HT}
For any $v\in\R^p$ and $s_\star$-sparse $w\in\R^p$, if $s\geq s_\star$, then
\[\inner{v - \HT(v)}{w - \HT(v)} \leq \frac{\sqrt{s_\star}}{2\sqrt{s}}\norm{w - \HT(v)}^2_2.\]
\end{lemma}

\begin{lemma}\label{lem:ID}
For any vectors $v,w\in\R^n$ and any $L$-Lipschitz monotone nondecreasing function $g$,
\[\inner{v}{ g(v) - \iso_{w}(g(v))} \geq L^{-1}\norm{g(v) - \iso_{w}(g(v)) }^2_2.\]
\end{lemma}

\begin{proof}[Proof of Lemma~\ref{lem:ID}]
Let $a_1<\dots<a_K$ be the unique values of the vector $\iso_{w}(g(v))$, and for each $k=1,\dots,K$ let 
\[I_k = \Big\{i: \big( \iso_{w}(g(v))  \big)_i = a_k\Big\}.\] 
By definition of isotonic regression, we have
\begin{equation}\label{eqn:iso_mean}\frac{1}{|I_k|}\sum_{i\in I_k} g(v)  = a_k.\end{equation}
For each $k$, let $b_k\in\R$ satisfy
\[g(b_k) = a_k,\]
which must exist since $g:\R\rightarrow\R$ is continuous (due to being Lipschitz) and $a_k$ is a convex combination of values taken by $g$.
We then have
\begin{align*}
 \inner{v}{ g(v) - \iso_{w}(g(v)) } 
&=\sum_{i=1}^n v_i \cdot \big(g(v_i) -  \iso_{w}(g(v))_i\big)\\
&= \sum_{k=1}^K \sum_{i\in I_k} v_i \cdot \big(g(v_i) - a_k\big)\\
&= \sum_{k=1}^K\left( \sum_{i\in I_k} v_i \cdot \big(g(v_i) - a_k\big) - b_k \cdot \sum_{i\in I_k}  \big(g(v_i) - a_k\big) \right)\textnormal{ by~\eqref{eqn:iso_mean}}\\ 
&= \sum_{k=1}^K \sum_{i\in I_k} (v_i  - b_k)\cdot \big(g(v_i) - g(b_k)\big)\textnormal{ by definition of $b_k$} \\
&\geq \sum_{k=1}^K \sum_{i\in I_k} L^{-1} \big(g(v_i) - a_k\big)^2\textnormal{ since $g$ is $L$-Lipschitz and monotone}\\
&= \sum_{k=1}^K \sum_{i\in I_k} L^{-1} \Big(g(v_i) -\big( \iso_{w}(g(v)) \big)_i \Big)^2\textnormal{ by definition of $a_k$}\\
&=L^{-1} \norm{g(v) - \iso_{w}(g(v))}^2_2.
\end{align*}
\end{proof}

\subsection{Proof of Lemma~\ref{lem:bound_err_terms}}
First we bound $\textnormal{Err}_\infty(\Z)$. 
By the sparse eigenvalue bound~\eqref{eqn:sparse_eig} on $\X$, we have $\norm{\X_j}_2 = \norm{\X\mathbf{e}_j}_2 \leq \sqrt{\beta n}$
for all $j=1,\dots,p$.
Since $\Z=\Y - \mu_\star$, $\Z - \bar{Z}\ones_n$ is zero-mean and $\sigma$-subgaussian, we therefore
see that $\X_j^\top (\Z - \bar{Z}\ones_n)$ is also zero-mean and is $\sigma\sqrt{\beta n}$-subgaussian.
Therefore, for each $j$,
\[\PP{n^{-1}\big|\X_j^\top (\Z - \bar{Z}\ones_n)\big|> \sigma\sqrt{ \frac{2\beta \log(4p/\delta)}{n}}} \leq \frac{\delta}{2p},\]
and so taking a union bound, we have shown that
\[\PP{\textnormal{Err}_\infty(\Z) \leq \sigma\sqrt{ \frac{2\beta \log(4p/\delta)}{n}}}\geq 1- \frac{\delta}{2}.\]

Next we turn to $\textnormal{Err}_\iso(\Z)$. First,
for any vector $v\in\R^n$, define
 $\pi_v$ to be a permutation of $\{1,\dots,n\}$ such that $v_{\pi_v(1)}\leq \dots \leq v_{\pi_v(n)}$, i.e., we
rearrange the indices into the ordering of $v$. (If $v$ has some repeated entries then $\pi_v$ will not be unique but
we can arbitrarily choose one such permutation for each $v$.)

First, for any fixed $v$,
we will compute a deterministic bound on
$\norm{\iso_v(\Y) - \iso_v(\mu_\star)}_2$. 
We know that $\mu_\star\in[-B,B]^n$ and therefore $\iso_v(\mu_\star)\in[-B,B]^n$,
and is monotone nondecreasing by definition. 
By Lemma 11.1 in \cite{chatterjee2015}, we can find some partition 
\[\{1,\dots,n\} = \underbrace{\{\pi_v(1),\dots,\pi_v(k_1)\}}_{=:S_1} \cup \underbrace{\{\pi_v(k_1+1),\dots,\pi_v(k_2)\}}_{=:S_2}\cup\dots\cup
 \underbrace{\{\pi_v(k_{M-1}+1),\dots,\pi_v(n)\}}_{=:S_M},\]
with $M\leq \lceil n^{1/3}\rceil$, such that 
\[\max_{i\in S_m} \iso_v(\mu_\star) - \min_{i\in S_m}\iso_v(\mu_\star)\leq 2Bn^{-1/3}\textnormal{ for all $m=1,\dots,M$}.\]
In other words, this condition ensures $\iso_v(\mu_\star)$ has little variation within each index subset $S_m$.
We define $k_0 = 0$ and $k_M=n$ below, so that we can write $S_m = \{\pi_v(k_{m-1}+1),\dots,\pi_v(k_m)\}$ for each $m=1,\dots,M$.

Next for any permutation $\pi$ of $\{1,\dots,n\}$, define a seminorm 
\[\norm{w}_{\textnormal{SW},\pi} = \max_{1\leq j \leq k \leq n} \frac{\left|\sum_{\ell=j}^k w_{\pi(\ell)}\right|}{\sqrt{k-j+1}}.\]
This is the ``sliding window'' norm of \cite{yang2018}, defined according to the ordering of the vector $v$.
By Theorem 1 and Lemma 2 of \cite{yang2018}, it holds that
\[\norm{\iso_v(x) - \iso_v(x')}_{\textnormal{SW},\pi_v} \leq \norm{x-x'}_{\textnormal{SW},\pi_v}\]
for all $x,x'\in\R^n$. Therefore, for all $1\leq j\leq k\leq n$,
\[\frac{\left|\sum_{\ell = j}^k \big(\iso_v(\Y)_{\pi_v(\ell)} - \iso_v(\mu_\star)_{\pi_v(\ell)}\big)\right|}{ \sqrt{k-j+1}}
\leq \norm{\iso_v(\Y)-\iso_v(\mu_\star)}_{\textnormal{SW},\pi_v}
\leq \norm{\Y - \mu_\star}_{\textnormal{SW},\pi_v} =  \norm{\Z}_{\textnormal{SW},\pi_v}.\]
Now fix any $\ell\in\{1,\dots,n\}$ and let $m_{\ell}$ be the index such that $\ell\in S_{m_\ell}$.
We have
\begin{align*}\iso_v(\Y)_{\pi_v(\ell)}
& \leq \frac{\sum_{j=\ell}^{k_{m_\ell}} \iso_v(\Y)_{\pi_v(j)}}{k_{m_\ell}-\ell+1}\textnormal{\quad since $\iso_v(\Y)$ is monotone nondecreasing}\\
& \leq \frac{\sum_{j=\ell}^{k_{m_\ell}} \iso_v(\mu_\star)_{\pi_v(j)}}{k_{m_\ell}-\ell+1} + \frac{\norm{\Z}_{\textnormal{SW},\pi_v}}{\sqrt{k_m-\ell+1}}\textnormal{\quad by the bound above}\\
& \leq \iso_v(\mu_\star)_{\pi_v(\ell)} + 2Bn^{-1/3} + \frac{\norm{\Z}_{\textnormal{SW},\pi_v}}{\sqrt{k_{m_\ell}-\ell+1}}\textnormal{\quad by construction of the partition}.
\end{align*}
Similarly,
\[\iso_v(\Y)_{\pi_v(\ell)} \geq \iso_v(\mu_\star)_{\pi_v(\ell)}-  2Bn^{-1/3} - \frac{\norm{\Z}_{\textnormal{SW},\pi_v}}{\sqrt{\ell - k_{{m_\ell}-1}}}.\]
Therefore,
\[\left|\iso_v(\Y)_{\pi_v(\ell)} - \iso_v(\mu_\star)_{\pi_v(\ell)}\right| \leq   2Bn^{-1/3} +  \frac{\norm{\Z}_{\textnormal{SW},\pi_v}}{\sqrt{\min\{ \ell - k_{m_\ell-1},k_{m_\ell} -\ell + 1\}}},\]
and so
\begin{align*}
\norm{\iso_v(\Y) - \iso_v(\mu_\star)}_2
& \leq \sqrt{n} \cdot 2Bn^{-1/3} + \norm{\Z}_{\textnormal{SW},\pi_v}\cdot  \sqrt{\sum_{\ell=1}^n \frac{1}{\min\{\ell - k_{m_\ell-1},k_{m_\ell} -\ell + 1\}}}\\
& =  2Bn^{1/6} + \norm{\Z}_{\textnormal{SW},\pi_v}\cdot  \sqrt{\sum_{m=1}^M \sum_{\ell=1}^{k_m - k_{m-1}} \frac{1}{\min\{\ell,k_m-k_{m-1}-\ell+1\}}}\\
& \leq 2Bn^{1/6} +  \norm{\Z}_{\textnormal{SW},\pi_v}\sqrt{\sum_{m=1}^M 2\log\big(2(k_m-k_{m-1})\big)}\\
&  = 2Bn^{1/6} +  \norm{\Z}_{\textnormal{SW},\pi_v}\sqrt{2\log\big(2^M(k_1-k_0)(k_2-k_1)\dots(k_M-k_{M-1})\big)}.
\end{align*}
Since $0=k_0 < k_1 < \dots < k_M=n$,  it holds that $(k_1-k_0)(k_2-k_1)\dots(k_M-k_{M-1}) \leq (n/M)^M$, and therefore,
\[\norm{\iso_v(\Y) - \iso_v(\mu_\star)}_2
 \leq  2Bn^{1/6} +  \norm{\Z}_{\textnormal{SW},\pi_v}\sqrt{2M\log(2n/M)}.\]
 Since $M=\lceil n^{1/3}\rceil$, as long as $n\geq 2$ we can relax this to
 \[\norm{\iso_v(\Y) - \iso_v(\mu_\star)}_2\leq 2Bn^{1/6} +  \norm{\Z}_{\textnormal{SW},\pi_v} 2n^{1/6}\sqrt{\log n}.\]
 
 Since this holds deterministically for any $v$, by considering $v=\X u$ we see that
\begin{multline*}\textnormal{Err}_\iso(\Z) = \sup_{u \in \mathbb{S}_s^{p-1}}\left\{n^{-1/2}\norm{\iso_{\X u}(\Y) -  \iso_{\X u}(\mu_\star)}_2 \right\}\\
\leq 2Bn^{-1/3} + 2n^{-1/3}\sqrt{\log n}\cdot \sup_{u \in \mathbb{S}_s^{p-1}}\left\{ \norm{\Z}_{\textnormal{SW},\pi_{\X u}} \right\}.\end{multline*}

Next
define
\[\mathcal{S}_{n,p}^{\textnormal{$s$-sparse}}(X_1,\dots,X_n) = \left\{\pi \in \mathcal{S}_n : \pi = \pi_{\X u}\textnormal{ for some $u \in \mathbb{S}_s^{p-1}$} \right\}.\] 
In other words, this is the set of possible orderings of the entries of $\X u$, under any $u \in \mathbb{S}_s^{p-1}$. 
Write $\mathcal{S}_{n,p}^{\textnormal{$s$-sparse}}(X_1,\dots,X_n) = \{\pi_1,\dots,\pi_K\}$ where $K=\left|\mathcal{S}_{n,p}^{\textnormal{$s$-sparse}}(X_1,\dots,X_n) \right|$. 
This means that
for every $s$-sparse $u$, 
\[\norm{\Z}_{\textnormal{SW},\pi_{\X u}}=\norm{\Z}_{\textnormal{SW},\pi_k}\textnormal{ for some $k\in\{1,\dots,K\}$},\]
and so we can write
\[\textnormal{Err}_\iso(\Z) 
\leq 2Bn^{-1/3} + 2n^{-1/3}\sqrt{\log n}\cdot \max_{k=1,\dots,K}\norm{\Z}_{\textnormal{SW},\pi_k}.\]
Finally, for $\sigma$-subgaussian zero-mean $\Z$ and any fixed permutation $\pi$, Lemma 3 of \cite{yang2018} proves that
 \[ \norm{\Z}_{\textnormal{SW},\pi}\leq \sqrt{2\sigma^2\log\left(\frac{n^2+n}{\delta'}\right)}\]
 with probability at least $1-\delta'$, for any $\delta'>0$. Setting $\delta' = \frac{\delta}{2K}$, then,
 \[\max_{k=1,\dots,K} \norm{\Z}_{\textnormal{SW},\pi_k}\leq \sqrt{2\sigma^2\log\left(\frac{(n^2+n)\cdot 2K}{\delta}\right)}\]
 with probability at least $1-\frac{\delta}{2}$.
Finally, Lemma~\ref{lem:perms} below proves that, deterministically,
\[K=\left|\mathcal{S}_{n,p}^{\textnormal{$s$-sparse}}(X_1,\dots,X_n) \right|\leq n^{2s-1}p^s.\]
Therefore, with probability at least $1-\frac{\delta}{2}$,
 \[\max_{k=1,\dots,K} \norm{\Z}_{\textnormal{SW},\pi_k}\leq \sqrt{2\sigma^2\log\left(\frac{(n^2+n)\cdot 2n^{2s-1}p^s}{\delta}\right)},\]
 which completes the proof.

\subsubsection{Supporting lemma}
\begin{lemma}\label{lem:perms}
For any sample size $n$, dimension $p$,  sparsity level $s$, and points $x_1,\dots,x_n\in\R^p$, define the set
\[\mathcal{S}_{n,p}^{\textnormal{$s$-sparse}}(x_1,\dots,x_n) = \left\{\pi \in \mathcal{S}_n : x_{\pi(1)}^\top u \leq x_{\pi(2)}^\top u\leq \dots \leq x_{\pi(n)}^\top u\textnormal{ for some $s$-sparse $u\in\R^p$}\right\}.\] 
Then we
have
\[\left|\mathcal{S}_{n,p}^{\textnormal{$x$-sparse}}(x_1,\dots,x_n) \right| \leq n^{2s-1}p^s.\]
\end{lemma}
\begin{proof}
Define 
\[\mathcal{S}_{n,s}(x_1,\dots,x_n) = \left\{\pi \in \mathcal{S}_n : x_{\pi(1)}^\top u \leq x_{\pi(2)}^\top u\leq \dots \leq x_{\pi(n)}^\top u\textnormal{ for some $u\in\R^s$}\right\}.\] 
First, for any $s$-sparse $u\in\R^p$, let $A\subseteq[p]$ be some subset of size $|A|=s$ containing the support of $u$. Denote the permutation $\pi$ such that $x_{\pi(1)}^\top u \leq x_{\pi(2)}^\top u\leq \dots \leq x_{\pi(n)}^\top u$ as $\pi_u(x_1,\dots,x_n)$. Then it's clear that
\[\pi_u(x_1,\dots,x_n) \in \mathcal{S}_{n,s}(x_{1A},\dots,x_{nA}).\]
where we write $x_{iA}\in\R^s$ as the subvector of $x_i$ with entries indexed by $A$. Therefore we have 
\[\left|\mathcal{S}_{n,p}^{\textnormal{$s$-sparse}}(x_1,\dots,x_n) \right| \leq \sum_{A\subset[p],|A|\leq s}\left|\mathcal{S}_{n,s}(x_{1A},\dots,x_{nA})\right| \leq p^s \cdot \max_{y_1,\dots,y_n\in\R^s}\left|\mathcal{S}_{n,s}(y_1,\dots,y_n)\right|,\]
since the number of subsets $A$ is bounded by ${p\choose s}\leq p^s$. Therefore, it suffices to compute a bound for the non-sparse case, i.e.~to bound $\left|\mathcal{S}_{n,s}(y_1,\dots,y_n)\right|$. Cover (1967) proves that for points $y_1,\dots,y_n$ in generic position, this is exactly equal to $Q(n,s)$ which is defined by the recursive relation
\[Q(n+1,s) = \begin{cases}1, & n=0,\\ Q(n,s) + n Q(n,s-1), & s\geq 2,~\text{and}~n\geq 1\\ 2,& s = 1~\text{and}~n\geq 1.\end{cases}\]
If $y_1,\dots,y_n$ are not in generic position, then the cardinality is instead bounded by $Q(n,s)$. Finally, using this recursive relation, we check that for the case $n=1$, trivially $Q(1,s) = 1 \leq 1^{2s-1}$ for any $s\geq 1$, and for the case $s=1$ and $n\geq 2$, $Q(n,1) = 2 \leq n^{2\cdot 1-1}$. Then for $n\geq 2$ and $s\geq 2$, we inductively have
\[Q(n+1,s) = Q(n,s) + nQ(n,s-1)\leq n^{2s-1} + n\cdot n^{2(s-1)-1} = (n+1)\cdot n^{2s-2} \leq (n+1)^{2s-1},\]
proving the desired bound.
\end{proof}

\subsection{Proof of Lemma~\ref{lem:init}}

Define $\tilde{u}_0 = \frac{1}{n}\X^\top (\Y - \bar{Y}\ones_{n})$. 
First, writing $\bar{g}_\star(\X u_\star) = \frac{1}{n}\sum_{i=1}^n g_\star(X_i^\top u_\star)$, we have 
\begin{align*}
\inner{\tilde{u}_0}{u_\star} 
&=  \frac{1}{n}\inner{u_\star}{\X^\top(\Y - \bar{Y}\ones_{n})}\\
&=  \frac{1}{n}\inner{\X u_\star}{g_\star(\X u_\star) - \bar{g}_\star(\X u_\star)\ones_{n}} + \frac{1}{n}\inner{u_\star}{\X^\top(\Z - \bar{Z} \ones_{n})}\\
&\geq  \frac{1}{n}\inner{\X u_\star}{g_\star(\X u_\star) - \bar{g}_\star(\X u_\star)\ones_{n}} - \sqrt{s_\star} \norm{n^{-1}\X^\top(\Z - \bar{Z}\ones_{n})}_{\infty}\\
&\geq \frac{1}{nL} \norm{g_{\star}(\X u_{\star})-\bar{g}_{\star}(\X u_{\star}) \ones_n }^2_2 - \sqrt{s_\star}\norm{ n^{-1} \X^\top (\Z-\bar{Z}\ones_n)}_{\infty},
\end{align*}
where the first inequality holds since
\[\inner{u_\star}{\X^\top(\Z - \bar{Z} \ones_{n})}
\leq \norm{u_\star}_1 \norm{\X^\top(\Z - \bar{Z} \ones_{n})}_{\infty}\]
and $\norm{u_\star}_1\leq\sqrt{s_\star}\norm{u_\star}_2=\sqrt{s_\star}$ since $u_\star$ is a $s_\star$-sparse unit vector,
while the second inequality can be seen by applying Lemma \ref{lem:ID} with $w=\ones_n$ and $g=g_\star$.
Next, we have
\begin{multline*} \frac{1}{L^2n} \norm{g_{\star}(\X u_{\star})-\bar{g}_{\star}(\X u_{\star}) \ones_n }^2_2
\geq \frac{1}{L^2n} \norm{g_{\star}(\X u_{\star})-\iso_{-\X u}(g_{\star}(\X u_\star))}^2_2\\ \geq \alpha\norm{u_\star-(-u_\star)}^2_2 = 4\alpha,\end{multline*}
where the first step holds since trivially, the constant vector $\bar{g}_{\star}(\X u_{\star}) \ones_n$ agrees with the ordering of $-\X u_\star$, while the second
step holds by assumption~\eqref{eqn:ID} applied with $u=-u_\star$. In particular, by the assumption of the lemma, 
we see that $\inner{\tilde{u}_0}{u_\star}>0$, verifying that $\tilde{u}_0\neq 0$. Note that, by definition, we 
have
\[u_0 =  \frac{\HT(\tilde{u}_0)}{ \norm{\HT(\tilde{u}_0)}_2 }.\]

Next, we apply Lemma \ref{lem:HT} with $v = \frac{\tilde{u}_0}{\norm{\HT(\tilde{u}_0)}_2}$ and $w=u_\star$.
Noting that $\HT(v) = \frac{\HT(\tilde{u}_0)}{\norm{\HT(\tilde{u}_0)}_2}=u_0$,
this lemma yields
\[\biginner{\frac{\tilde{u}_0}{\norm{\HT(\tilde{u}_0)}_2}-u_0}{u_\star-u_0}\leq \frac{\sqrt{s_\star}}{2\sqrt{s}}\norm{u_\star - u_0}^2_2 = \sqrt{\frac{s_\star}{s}}\left(1 - \inner{u_0}{u_\star}\right),\]
where the last step holds since $u_0$ and $u_\star$ are both unit vectors.
We also have
\[\biginner{\frac{\tilde{u}_0}{\norm{\HT(\tilde{u}_0)}_2}-u_0}{u_0}=0,\]
since these two vectors have disjoint support by definition. Therefore,
\[\biginner{\frac{\tilde{u}_0}{\norm{\HT(\tilde{u}_0)}_2}-u_0}{u_\star}\leq \sqrt{\frac{s_\star}{s}}\left(1 - \inner{u_0}{u_\star}\right),\]
and after rearranging terms, we have
\[\frac{1}{\norm{\HT(\tilde{u}_0)}_2} \inner{\tilde{u}_0}{u_\star} \leq \sqrt{\frac{s_\star}{s}} + \left(1 - \sqrt{\frac{s_\star}{s}}\right)\inner{u_0}{u_\star}.\]
Combining this with the calculations above, we have proved that
\[
 \inner{u_0}{u_\star}
\geq \left(1-\sqrt{\frac{s_\star}{s}}\right)^{-1}\left(\frac{1}{\norm{\HT(\tilde{u}_0)}_2} \inner{\tilde{u}_0}{u_\star} - \sqrt{\frac{s_\star}{s}}\right).\]
Combining this with our lower bound on $\inner{\tilde{u}_0}{u_\star}$ calculated above, we now have
\begin{multline*}
 \inner{u_0}{u_\star}
\geq \frac{1}{\norm{\HT(\tilde{u}_0)}_2}\cdot \left(1-\sqrt{\frac{s_\star}{s}}\right)^{-1}\cdot\\ \left( \frac{1}{nL} \norm{g_{\star}(\X u_{\star})-\bar{g}_{\star}(\X u_{\star}) \ones_n }^2_2 - \sqrt{s_\star}\norm{ n^{-1} \X^\top (\Z-\bar{Z} \ones_n)}_{\infty} - \sqrt{\frac{s_\star}{s}} \cdot {\norm{\HT(\tilde{u}_0)}_2}\right).\end{multline*}
Next, writing $S\subseteq\{1,\dots,p\}$ to 
denote the support of $u_0$ (with $|S|\leq s$), we calculate
\begin{multline*}
\norm{\HT(\tilde{u}_0)}_2
=\norm{n^{-1}\X_S^\top(\Y-\bar{Y}\ones_n)}_2\\
\leq \norm{n^{-1}\X_S^\top(g_\star(\X u_\star) - \bar{g}_\star(\X u_\star))}_2 + \norm{n^{-1}\X_S^\top(\Z-\bar{Z}\ones_n)}_2\\
\leq \norm{n^{-1}\X_S^\top(g_\star(\X u_\star) - \bar{g}_\star(\X u_\star))}_2
 + \sqrt{s}\cdot \norm{n^{-1}\X^\top(\Z-\bar{Z}\ones_n)}_{\infty},
\end{multline*}
and so we now have
\begin{multline*}
 \inner{u_0}{u_\star}
\geq \frac{1}{\norm{\HT(\tilde{u}_0)}_2}\cdot \left(1-\sqrt{\frac{s_\star}{s}}\right)^{-1}\cdot\\ \bigg( \frac{1}{nL} \norm{g_{\star}(\X u_{\star})-\bar{g}_{\star}(\X u_{\star}) \ones_n }^2_2 - 2\sqrt{s_\star}\norm{ n^{-1} \X^\top (\Z-\bar{Z}\ones_n)}_{\infty} \\ - \sqrt{\frac{s_\star}{s}} \cdot \norm{n^{-1}\X_S^\top(g_\star(\X u_\star) - \bar{g}_\star(\X u_\star)\ones_n)}_2 \bigg).\end{multline*}
Furthermore,
\begin{align*}
\norm{n^{-1}\X_S^\top(g_\star(\X u_\star) - \bar{g}_\star(\X u_\star)\ones_n)}_2
&=\sup_{u\in\R^p: \norm{u}_2\leq 1, \text{support}(u)\subseteq S}\left|n^{-1}u^\top\X^\top(g_\star(\X u_\star)-\bar{g}_\star(\X u_\star)\ones_n)\right|\\
&\leq\sup_{u\in\R^p: \norm{u}_2\leq 1, \text{support}(u)\subseteq S}\norm{n^{-1}\X u}_2\norm{(g_\star(\X u_\star)-\bar{g}_\star(\X u_\star)\ones_n)}_2\\
&\leq \sqrt{\frac{\beta}{n}}\cdot \norm{g_\star(\X u_\star)-\bar{g}_\star(\X u_\star)\ones_n}_2,
\end{align*}
where the last step holds by assumption~\eqref{eqn:sparse_eig}. Combining everything, we have shown that
\begin{multline*}
 \inner{u_0}{u_\star}
\geq \frac{1}{\norm{\HT(\tilde{u}_0)}_2}\cdot \left(1-\sqrt{\frac{s_\star}{s}}\right)^{-1}\cdot\\ \bigg( \frac{1}{nL} \norm{g_{\star}(\X u_{\star})-\bar{g}_{\star}(\X u_{\star}) \ones_n }^2_2 - 2\sqrt{s_\star}\norm{ n^{-1} \X^\top (\Z-\bar{Z}\cdot \mathbf{1})}_{\infty} \\ - \sqrt{\frac{\beta s_\star}{s n}} \cdot \norm{g_{\star}(\X u_{\star})-\bar{g}_{\star}(\X u_{\star}) \ones_n }_2 \bigg).\end{multline*}

Next, recalling that we have shown $\frac{1}{L^2n} \norm{g_{\star}(\X u_{\star})-\bar{g}_{\star}(\X u_{\star}) \ones_n }^2_2\geq 4\alpha$, since we've assumed that $s\geq s_\star \cdot \frac{\beta}{\alpha}$ we have
\begin{multline*}
\frac{1}{nL} \norm{g_{\star}(\X u_{\star})-\bar{g}_{\star}(\X u_{\star}) \ones_n }^2_2- \sqrt{\frac{\beta s_\star}{s n}} \cdot \norm{g_{\star}(\X u_{\star})-\bar{g}_{\star}(\X u_{\star}) \ones_n }_2\\
\geq \frac{1}{nL} \norm{g_{\star}(\X u_{\star})-\bar{g}_{\star}(\X u_{\star}) \ones_n }^2_2- \sqrt{\frac{\alpha}{n}}\cdot \norm{g_{\star}(\X u_{\star})-\bar{g}_{\star}(\X u_{\star}) \ones_n }_2
\geq 2\alpha L.
\end{multline*}
Therefore,
\[
 \inner{u_0}{u_\star}
\geq \frac{1}{\norm{\HT(\tilde{u}_0)}_2}\cdot \left(1-\sqrt{\frac{s_\star}{s}}\right)^{-1}\cdot\\ \left( 2\alpha L - 2\sqrt{s_\star}\norm{ n^{-1} \X ^\top(\Z-\bar{Z}\ones_n)}_{\infty}  \right).\]
This completes the proof.

\subsection{Proofs for Proposition \ref{thm:Pred_err}} \label{pf:Pred_err}

By definition of isotonic regression,
\begin{eqnarray*}
&&n^{-1/2}\norm{\iso_{\X u_t}\Y - g_\star(\X u_\star)}_2\\ &\leq& n^{-1/2}\norm{\iso_{\X u_t}\Y - \iso_{\X u_t}g_\star(\X u_\star) }_2 + n^{-1/2}\norm{\iso_{\X u_t}g_\star(\X u_\star) - g_\star(\X u_\star)}_2 \\
&\leq&  n^{-1/2}\norm{\iso_{\X u_t}\Y - \iso_{\X u_t}g_\star(\X u_\star) }_2 +  n^{-1/2}\norm{g_\star(\X u_t) - g_\star(\X u_\star)}_2\\
&\leq& \textnormal{Err}_\iso (\Z)  + L\sqrt{\beta}\norm{u_t-u_\star}_2,
\end{eqnarray*}
where the second inequality is from the contractive property of isotonic regression in $\ell_2$-norm (see \citet{yang2018}), and $\textnormal{Err}_\iso (\Z)$ is defined as in the proof of Theorem~\ref{thm:converge_main}. We finish the proof by plugging in the results from Lemma \ref{lem:bound_err_terms} and Theorem \ref{thm:converge_main}:
\begin{align*}
    & n^{-1/2}\norm{\iso_{\X u_t}\Y - g_\star(\X u_\star)}_2 \\
    \leq & L\sqrt{2\beta}\cdot r^t + C(L\sqrt{\beta} + 4\sigma) \left(\eps_n + \frac{s^{1/2}\log(np/\delta)}{n^{1/3}}\right) ~\text{with probability at least $1-\delta$}.
\end{align*}

\subsection{Proofs for the normal mixture $\X$ (Proposition \ref{prop:converge_normal_mix})} \label{pf:rand_X}

In this section we provide the proof for Proposition \ref{prop:converge_normal_mix}. The main body of the proof is to verify that under the assumptions in section \ref{sec:assumption_rand}, conditions  \eqref{eqn:sparse_eig} and \eqref{eqn:ID} are satisfied with high probability. We show this in the following two lemmas.  

For the upper bound on sparse eigenvalue condition, we have the following result: 
\begin{lemma}\label{thm:eig_random}
Assuming that the rows of $\X$ are \iid~draws from a distribution on $X\in\R^p$, with second moment $\Sigma =\EE{XX^\top}$ satisfying
\[c_0 \ident_p \preceq \Sigma \preceq c_1\ident_p ~\text{for some constants $c_1\geq c_0>0$,}\] ~then for all $\delta >0$, as $n$ is large enough, \eqref{eqn:sparse_eig} is satisfied with probability at least $1-\delta/2$ with $\beta=2c_1$.
\end{lemma}

For the lower bound on sparse monotone single index eigenvalues, we have the following lemma:

\begin{lemma}\label{thm:ID_random}
Under assumptions \eqref{eqn:lipschitz},  \eqref{eqn:randomX_normal}, \eqref{eqn:randomX_var}, for all $\delta>0$, as $n$ is sufficiently large, there exists $\epsilon_n \geq C_\alpha \cdot \sqrt{ \frac{s\log(np/\delta)}{n}}$ that ~\eqref{eqn:ID} is satisfied with probability at least $1-\delta/2$, with $\alpha>0$ and $C_\alpha>0$ being constants depending on $L,c_0,c_1,B,\nu,V$ but not on $n,p,s,\delta$. 
\end{lemma}

Combining Lemmas \ref{thm:eig_random} and \ref{thm:ID_random} we finish the proof of Proposition \ref{prop:converge_normal_mix}. The proof of Lemmas \ref{thm:eig_random} and \ref{thm:ID_random} are deferred in section \ref{pf:eig_random}.

\subsubsection{Proof of  Lemmas \ref{thm:eig_random} and \ref{thm:ID_random}}\label{pf:eig_random}
\begin{proof}[Proof of Lemma \ref{thm:eig_random}]
For any fixed support $S\subset[p]$ with $|S| = 2s + s_\star$, by 5.40 of \cite{vershynin2012},
\[\PP{\norm{\frac{1}{n}\X_S^\top\X_S - \cov(X_S)}\leq \sigma^2 C\left( \sqrt{\frac{2s+s_\star + \log(1/\delta)}{n}} +\frac{2s+s_\star + \log(1/\delta)}{n}\right)}  \geq 1- \delta/2,\]
for a universal constant $C$.
Taking a union bound over all ${p\choose 2s+s_\star}$ possible supports $S$ of size $|S|=2s+s_\star$,
\begin{multline*}\mathbb{P}\left\{\sup_S\norm{\frac{1}{n}\X_S^\top\X_S - \cov(X_S)} \right.\\
\left. \leq \sigma^2 C\left( \sqrt{\frac{2s+s_\star +(2s+s_\star ) \log(p/\delta)}{n}} +\frac{2s+s_\star + (2s+s_\star ) \log(p/\delta)}{n}\right)\right\} \geq 1- \delta/2.\end{multline*}
Simplifying,
\[\PP{\sup_S\norm{\frac{1}{n}\X_S^\top\X_S - \cov(X_S)}\leq 6\sigma^2 C\left( \sqrt{\frac{s \log(p/\delta)}{n}} +\frac{s\log(p/\delta)}{n}\right)}  \geq 1- \delta/2.\]
We can rewrite this as
\[\PP{\sup_{u \in \mathbb{S}^{p-1}_s}\frac{1}{n}\norm{\X u}^2_2 \leq c_1 + 6\sigma^2 C\left( \sqrt{\frac{s \log(p/\delta)}{n}} +\frac{s\log(p/\delta)}{n}\right)}  \geq 1- \delta/2,\]
where 
\[\mathbb{S}^{p-1}_s = \left\{u\in\R^p : \norm{u}_2=1,\ \textnormal{$u$ is $s$-sparse}\right\}\]
is the set of $s$-sparse unit vectors in $\R^p$.
Let $\beta = 2c_1$, then with $n$ sufficiently large, 
\[\PP{\sup_{u \in\mathbb{S}^{p-1}_s}\frac{1}{n}\norm{\X u}^2_2 \leq \beta }  \geq 1- \delta/2.\]
This concludes the proof of Lemma \ref{thm:eig_random}.
\end{proof}

\begin{proof}[Proof of Lemma \ref{thm:ID_random}]
We would like to place a lower bound on $\frac{1}{L^2n}\norm{g(\X u) - g_\star(\X u_\star)}^2_2$ for all $u \in \mathbb{S}_s^{p-1}$ with $\norm{u-u_\star}_2 \geq \epsilon_n$ and all monotone non-decreasing functions $g:\R\rightarrow[-B,B]$, where $[-B,B]$ is the range of $g_\star$. 
%Let $\pr{[a,b]}(t) = \min\big\{\max\{t,a\},b\big\}$ project any scalar $t$ to the range $[a,b]$. Then $(t - g_\star(s))^2 \geq (\pr{[a,b]}(t)-g_\star(s))^2$ for any $s,t$, and therefore,
%\[\sqrt{\frac{1}{n}\norm{g(\X u) - g_\star(\X u_\star)}^2_2} \geq \sqrt{\frac{1}{n}\norm{\pr{[a,b]}\big(g(\X u)\big) - g_\star(\X u_\star)}^2_2}\]
%for any $g,u$. Now, for any non-decreasing $g$, the map $\pr{[a,b]}\circ g$, which maps $s\mapsto \pr{[a,b]}(g(s))$, is also non-decreasing. 
For $\alpha >0$ (we will specify $\alpha$ later in the proof), we have
\begin{align*}
&\inf_{\substack{u \in \mathbb{S}_s^{p-1}, \norm{u-u_\star}_2 \geq \epsilon_n\\\text{non-decr.~$g:\R\rightarrow[-B,B]$}}}\Bigg(\sqrt{\frac{1}{L^2n}\norm{g(\X u) - g_\star(\X u_\star)}^2_2}- \sqrt{\alpha}\norm{u-u_\star}_2\Bigg)\\
&=\inf_{\substack{u \in \mathbb{S}_s^{p-1}, \norm{u-u_\star}_2 \geq \epsilon_n\\\text{non-decr.~$g:\R\rightarrow[-B,B]$}}}\Bigg(\sqrt{\frac{1}{L^2n}\sum_i \big(g(X_i^\top u) - g_\star(X_i^\top u_\star)\big)^2}- \sqrt{\alpha}\norm{u-u_\star}_2\Bigg)\\
&\geq \inf_{\substack{u \in \mathbb{S}_s^{p-1}, \norm{u-u_\star}_2 \geq \epsilon_n\\\text{non-decr.~$g:\R\rightarrow[-B,B]$}}}\Bigg(\frac{1}{Ln}\sum_i \big|g(X_i^\top u) - g_\star(X_i^\top u_\star)\big|- \sqrt{\alpha}\norm{u-u_\star}_2\Bigg),
\end{align*}
where the last step holds by the standard inequality relating the $\ell_1$ and $\ell_2$ norms.
We can further decompose this as
\[\inf_{\substack{u \in \mathbb{S}_s^{p-1}, \norm{u-u_\star}_2 \geq \epsilon_n\\\text{non-decr.~$g:\R\rightarrow[-B,B]$}}}\Bigg(\sqrt{\frac{1}{L^2n}\norm{g(\X u) - g_\star(\X u_\star)}^2_2} - \sqrt{\alpha}\norm{u-u_\star}_2\Bigg)
\geq \text{Term 1} - \text{Term 2}\]
where
\begin{multline*}\text{Term 1} = \inf_{\substack{u \in \mathbb{S}_s^{p-1}, \norm{u-u_\star}_2 \geq \epsilon_n\\\text{non-decr.~$g:\R\rightarrow[-B,B]$}}}\Bigg(\frac{1}{L}\EE{\big|g(X^\top u) - g_\star(X^\top u_\star)\big|} - \sqrt{\alpha}\norm{u-u_\star}_2\Bigg)\\
%= \inf_{\substack{\text{$s$-sparse unit $u$}\\\text{non-decr.~$g$}}}\Bigg(\EE{\big|g(X^\top u) - g_\star(X^\top u_\star)\big|}- \sqrt{\alpha}\norm{u-u_\star}_2\Bigg)
\end{multline*}
and
\[\text{Term 2} = \sup_{\substack{u \in \mathbb{S}_s^{p-1}, \norm{u-u_\star}_2 \geq \epsilon_n\\\text{non-decr.~$g:\R\rightarrow[-B,B]$}}}\bigg( \frac{1}{L}\EE{\big|g(X^\top u) - g_\star(X^\top u_\star)\big|} - \frac{1}{Ln}\sum_i \big|g(X_i^\top u) - g_\star(X_i^\top u_\star)\big|\bigg).
\]
For Term 1, by Lemma \ref{lem:normal_mix} \eqref{eqn:ID_random}, which states that for all $s$-sparse unit $u$ and any non-decreasing function $g:\R\rightarrow[-B,B]$ we have
\[\EE{\big|g(X^\top u) - g_\star(X^\top u_\star)\big|}^2\geq \alpha_\star\norm{u-u_\star}^2_2,\]
where $\alpha_\star >0$ is a constant which does not depend on $n,p,s,\eps_n$. Therefore, let $\alpha = \frac{\alpha_\star}{2L^2}$,
\[\text{Term 1} \geq \sqrt{\alpha}\cdot \epsilon_n. \]

For Term 2, we will use concentration bounds to control Term 2. First, since both $g$ and $g_\star$ take values in $[-B,B]$, and therefore
$\big|g(X_i^\top u) - g_\star(X_i^\top u_\star)\big|\in[0,2B]$ for any $X_i$ and any non-decreasing $g:\R\rightarrow[-B,B]$ and $u \in \mathbb{S}_s^{p-1}$, we can see that
 replacing a single $X_i$ with a different vector $X_i'$ can perturb the value of Term 2 by at most $\frac{2B}{n}$. Therefore by McDiarmid's inequality \cite{McDiarmid}, for any $\delta_1>0$,
\[\PP{\text{Term 2} > \EE{\text{Term 2}} +2B\cdot  \sqrt{\frac{\log(1/\delta_1)}{2n}}} \leq \delta_1.\]
Next, by the symmetrization lemma \cite{Symmetrization}, we have
\[\EE{\text{Term 2}} \leq 2\EE{\sup_{\substack{s \in \mathbb{S}_s^{p-1}\\\text{non-decr.~$g:\R\rightarrow[-B,B]$}}}\frac{1}{n}\sum_i \xi_i\big|g(X_i^\top u) - g_\star(X_i^\top u_\star)\big|},\]
where $\xi_i\iidsim\text{Unif}\{\pm 1\}$ are \iid~Rademacher variables, drawn independently from the $X_i$'s. Next, consider the composition
\[X \mapsto \big( g(X_i^\top u) - g_\star(X_i^\top u_\star)\big) \mapsto \big|g(X_i^\top u) - g_\star(X_i^\top u_\star)\big|.\]
The second map is the absolute value function, which is $1$-Lipschitz. By the Lipschitz inequality for Rademacher complexity, therefore,
\[\EE{\text{Term 2}} \leq 4\EE{\sup_{\substack{s \in \mathbb{S}_s^{p-1}\\\text{non-decr.~$g:\R\rightarrow[-B,B]$}}}\frac{1}{n}\sum_i \xi_i\big(g(X_i^\top u) - g_\star(X_i^\top u_\star)\big)} \leq 16B\sqrt{\frac{s\log(np)}{n}}.\]
The last inequality is because of Lemma \ref{lem:conc}.

Therefore, combining everything, with probability at least $1-\delta/2$,
\[\text{Term 2} \leq B \left(16\sqrt{\frac{s\log(np)}{n}} +  \sqrt{\frac{2\log(2/\delta)}{n}}\right).\]
With $n$ sufficiently large, let $\epsilon_n \geq \frac{16B}{\sqrt{\alpha}} \sqrt{\frac{s\log(np/\delta)}{n}}$, $\text{Term 2} \leq 16B \sqrt{\frac{s\log(np/\delta)}{n}} \leq \sqrt{\alpha}\cdot \epsilon_n$. This concludes the proof of Lemma \ref{thm:ID_random}.
\end{proof}

\subsubsection{Supporting lemmas}\label{pf:rand_supp}

\begin{lemma}\label{lem:normal_mix}
Suppose assumptions \eqref{eqn:lipschitz},  \eqref{eqn:randomX_normal}, and \eqref{eqn:randomX_var} hold, 
then 
\begin{equation}\label{eqn:ID_random}
\EE{\big|g(X^\top u) - g_\star(X^\top u_\star)\big|}^2\geq \alpha_\star\norm{u-u_\star}^2_2
\end{equation}
holds for all $s$-sparse unit vectors $u$, and any non-decreasing function $g: \R \rightarrow [-B,B]$. where $\alpha_\star>0$ is a constant depending on $L,c_0,c_1,B,\nu,V$ but not on $n,p,s$.
\end{lemma}

\begin{proof}
Let $A\in\{1,\dots,K\}$ be the latent variable indicating which component $X$ was drawn from. Fix any $u \in \mathbb{S}^{p-1}_s$. Then for any $g$,
\begin{multline*}\EE{\big|g(X^\top u) - g_\star(X^\top u_\star)\big|} \\= \sum_{k=1}^K a_k \EEst{\big|g(X^\top u) - g_\star(X^\top u_\star)\big|}{A=k} =  \sum_{k=1}^K a_k \EE{\big|g(R_k) - g_\star(S_k)\big|},\end{multline*}
where $(R_k,S_k)$ is bivariate normal with mean 
\[\left(\begin{array}{c}\mu_k^\top u\\\mu_k^\top u_\star\end{array}\right)\eqqcolon m,\]
and covariance
\[\left(\begin{array}{cc}u^\top\Sigma_k u & u^\top \Sigma_k u_\star\\  u_\star^\top \Sigma_k u&u_\star^\top \Sigma_k u_\star\end{array}\right).\]
Here we can calculate 
\[\rho =  \frac{u_\star^\top \Sigma_k u}{\sqrt{ u^\top \Sigma_k u\cdot  u_\star^\top \Sigma_k u_\star}} 
=  \frac{\frac{c}{2}u^\top \Sigma_k u + \frac{1}{2c}u_\star^\top \Sigma_k u_\star-\frac{1}{2}(c^{1/2}u-c^{-1/2}u_\star)^\top \Sigma_k(c^{1/2}u-c^{-1/2}u_\star)}{\sqrt{ u^\top \Sigma_k u\cdot  u_\star^\top \Sigma_k u_\star}} ,\]
for any $c>0$. Choosing $c = \sqrt{\frac{  u_\star^\top \Sigma_k u_\star}{u^\top \Sigma_k u}}$, we obtain
\[\rho 
=  1 - \frac{(c^{1/2}u-c^{-1/2}u_\star)^\top \Sigma_k(c^{1/2}u-c^{-1/2}u_\star)}{2\sqrt{ u^\top \Sigma_k u\cdot  u_\star^\top \Sigma_k u_\star}} \leq 1 - \frac{c_0}{2c_1}\norm{c^{1/2}u-c^{-1/2}u_\star}^2_2 ,\]
where $c_0$ and $c_1$ are the smallest and largest restricted eigenvalues of $\Sigma_k$. Next,
\[\norm{c^{1/2}u-c^{-1/2}u_\star}^2_2  = c\norm{u}^2_2 + c^{-1}\norm{u_\star}^2_2 - 2\inner{u}{u_\star} = c + c^{-1} - 2\inner{u}{u_\star} \geq 2 -2\inner{u}{u_\star}  = \norm{u-u_\star}^2_2,\]
since $u$ and $u_\star$ are both unit vectors. Combining everything,
\[\rho \leq 1 - \frac{c_0}{2c_1}\norm{u-u_\star}^2_2.\]
Note that $|m_2|\leq V$ since $\norm{\mu_k}_2\leq V$, and $u^\top\Sigma_k u,u_\star^\top \Sigma_k u_\star \geq c_0$ since  $ \Sigma_k\succeq c_0\ident_p$. 
Now suppose for the moment that $g_\star(T)-g_\star(-T)\geq C$ for some $T>0$ and $C>0$. Applying Lemma~\ref{lem:bivariate_normal} in the Appendix, we see that 
\[\EE{\big|g(R_k) - g_\star(S_k)\big|} \geq C_0\sqrt{1-(\rho\vee 0)^2},\]
where $C_0>0$ depends only on $L,C,T,c_0,V$,
for each $k=1,\dots,K$.  Plugging in our bound on $\rho$, we see that
\[\rho\vee 0 \leq  1 - \frac{c_0}{4c_1}\norm{u-u_\star}^2_2\]
must hold, since $c_0\leq c_1$ and $\norm{u-u_\star}^2_2\leq 4$ (since $u$ and $u_\star$ are unit vectors).
\[\EE{\big|g(R_k) - g_\star(S_k)\big|} \geq C_0 \sqrt{1 - \left( 1 - \frac{c_0}{4c_1}\norm{u-u_\star}^2_2\right)^2} \geq C_0 \sqrt{\frac{c_0}{2c_1}} \norm{u-u_\star}_2.\]
Since this is true for each $k$, therefore we have
\[\EE{\big|g(X^\top u) - g_\star(X^\top u_\star)\big|}  \geq C_0 \sqrt{\frac{c_0}{2c_1}} \norm{u-u_\star}_2.\]
 
Finally, we just need to find $T,C>0$ such that $g_\star(T)-g_\star(-T)\geq C$.
Recall that $\VV{g_\star(X^\top u_\star)}\geq \nu^2>0$ by assumption. Recall also that $g_\star$ takes values in $[-B,B]$.
 Define
\[T =V+\sqrt{c_1}\cdot  \Phi^{-1}\left(1-\frac{\nu^2}{16B^2}\right).\]
Then
\begin{multline*}\PP{|X^\top u_\star|>T} = \max_k\PP{\big|\mathcal{N}(u_\star^\top \mu_k,u_\star^\top\Sigma_ku_\star)\big|>T}\\
\leq \max_k \PP{|\mathcal{N}(0,1)| > \frac{T-|u_\star^\top \mu_k|}{\sqrt{u_\star^\top\Sigma_ku_\star}}} \leq \frac{\nu^2}{8B^2},\end{multline*}
by our bounds on $\mu_k$ and $\Sigma_k$ and our definition of $T$.
 Now let $C=g_\star(T)-g_\star(-T)$. Then
\begin{align*}
\nu^2& = \VV{g_\star(X^\top u_\star)}\\
& \leq \EE{(g_\star(X^\top u_\star)-g_\star(0))^2}\\
& =  \EE{(g_\star(X^\top u_\star)-g_\star(0))^2\cdot\One{|X^\top u_\star|>T}}+\EE{(g_\star(X^\top u_\star)-g_\star(0))^2\cdot\One{|X^\top u_\star|\leq T}}\\
& \leq   \EE{4B^2\cdot\One{|X^\top u_\star|>T}}+\EE{C^2}\\
&=4B^2\PP{|X^\top u_\star|>T} + C^2\\
&\leq 4B^2\cdot \frac{\nu^2}{8B^2} + C^2\\
&=\nu^2/2+C^2,
\end{align*}
and therefore,
we must have $C\geq \nu/\sqrt{2}$.
Combining everything,
\[\EE{\big|g(X^\top u) - g_\star(X^\top u_\star)\big|}  \geq \sqrt{\alpha_\star} \norm{u-u_\star}_2\]
for all $u \in \mathbb{S}^{p-1}_s$,
where $\alpha_\star>0$ is a constant depending on $L,c_0,c_1,B,\nu,V$.
\end{proof}

\begin{lemma}\label{lem:conc} 
With $g_\star: \R \rightarrow [a,b]$, and $b-a \leq 2B$,
\[\EE{\sup_{\substack{\text{$s$-sparse unit $u$}\\\text{non-decr.~$g:\R\rightarrow[a,b]$}}}\frac{1}{n}\sum_i \xi_i\big(g(X_i^\top u) - g_\star(X_i^\top u_\star)\big)}\leq4B \sqrt{\frac{s\log(np)}{n}}\]
\end{lemma}

\begin{proof}[Proof of Lemma~\ref{lem:conc}]
The expectation can be simplified as follows:
\begin{align*}
& \EE{\sup_{\substack{\text{$s$-sparse unit $u$}\\\text{non-decr.~$g:\R\rightarrow[a,b]$}}}\frac{1}{n}\sum_i \xi_i\big(g(X_i^\top u) - g_\star(X_i^\top u_\star)\big)}\\
=& \EE{\sup_{\substack{\text{$s$-sparse unit $u$}\\\text{non-decr.~$g:\R\rightarrow[a,b]$}}}\frac{1}{n}\sum_i \xi_i\cdot \big(g(X_i^\top u)  - a\big)}-4\EE{ \frac{1}{n}\sum_i \xi_i\cdot \big(g_\star(X_i^\top u_\star) - a \big)}\\
 =& \EE{\sup_{\substack{\text{$s$-sparse unit $u$}\\\text{non-decr.~$g:\R\rightarrow[a,b]$}}}\frac{1}{n}\sum_i \xi_i\cdot \big(g(X_i^\top u)  - a\big)}\\
 =& 2B\cdot \EE{\sup_{\substack{\text{$s$-sparse unit $u$}\\\text{non-decr.~$g:\R\rightarrow[0,1]$}}}\frac{1}{n}\sum_i \xi_i\cdot g(X_i^\top u)}.\end{align*}
where the next-to-last step holds since $\EE{\xi_i}=0$. 
Next, any non-decreasing function $g:\R\rightarrow[0,1]$ can be written as a (possibly infinite) convex combination of step functions. This means that for any $u$, the supremum over $g$ is attained at some step function, in other words, we can write
\begin{multline*}\EE{\sup_{\substack{\text{$s$-sparse unit $u$}\\\text{non-decr.~$g:\R\rightarrow[a,b]$}}}\frac{1}{n}\sum_i \xi_i\big(g(X_i^\top u) - g_\star(X_i^\top u_\star)\big)} \\ \leq 2B \cdot \EE{\sup_{\substack{\text{$s$-sparse unit $u$}\\t\in\R}}\frac{1}{n}\sum_i \xi_i\cdot \One{X_i^\top u \geq t}}.\end{multline*}
Now let $\pi_{u;X_1,\dots,X_n}$ be the permutation of $\{1,\dots,n\}$ induced by the ordering of the $X_i^\top u$'s, that is, the permutation $\pi\in\mathcal{S}_n$ that satisfies
\[X_{\pi(1)}^\top u \leq X_{\pi(2)}^\top u\leq \dots \leq X_{\pi(n)}^\top u.\]
To handle the possibility that this might not define a unique permutation, i.e.~since it might be the case that $\inner{X_i}{u} = \inner{X_{i'}}{u}$ for some $i\neq i'$, we place an arbitrary total ordering on $\mathcal{S}_n$ and then, in the case of ties, we define $\pi_{u;X_1,\dots,X_n}$ as the minimal permutation satisfying the ordering above.

Then the above can be rewritten as
\begin{multline*}\EE{\sup_{\substack{\text{$s$-sparse unit $u$}\\\text{non-decr.~$g:\R\rightarrow[a,b]$}}}\frac{1}{n}\sum_i \xi_i\big(g(X_i^\top u) - g_\star(X_i^\top u_\star)\big)} \\ \leq 2B \cdot \EE{\max\left\{0,\sup_{\substack{\text{$s$-sparse unit $u$}\\ k \in[n]}}\frac{1}{n}\sum_{i\geq k} \xi_{\pi_{u;X_1,\dots,X_n}(i)}\right\}}.\end{multline*}
Now define
\[\Pi_{p,s}(X_1,\dots,X_n) = \big\{\pi_{u;X_1,\dots,X_n} : \text{$s$-sparse unit $u\in\R^p$}\big\}\subseteq\mathcal{S}_n,\]
i.e.~the set of all permutations attained by any sparse vector $u$ (using the given feature vectors $X_1,\dots,X_n$). 
The above can therefore be rewritten as
\begin{multline*}\EE{\sup_{\substack{\text{$s$-sparse unit $u$}\\\text{non-decr.~$g:\R\rightarrow[a,b]$}}}\frac{1}{n}\sum_i \xi_i\big(g(X_i^\top u) - g_\star(X_i^\top u_\star)\big)} \\ \leq 2B \cdot \EE{\max\left\{0,\sup_{\substack{\pi\in \Pi_{p,s}(X_1,\dots,X_n)\\ k\in[n]}}\frac{1}{n}\sum_{i\geq k} \xi_{\pi(i)}\right\}}.\end{multline*}
By Lemma~\ref{lem:perms} in the Appendix, for any $X_1,\dots,X_n$ we have $|\Pi_{p,s}(X_1,\dots,X_n)|\leq n^{2s-1}p^s$. 
Now, for each $\pi$ and each $k$, $\frac{1}{n}\sum_{i\geq k} \xi_{\pi(i)}$ is a zero-mean subgaussian random variable at the scale $\frac{\sqrt{n-k+1}}{n}\leq n^{-1/2}$, and we are taking a supremum over $n^{2s-1}p^s\cdot n = n^{2s}p^s$ such variables. This expected value is therefore bounded as $\sqrt{\frac{2\log(2n^{2s}p^s)}{n}}\leq2 \sqrt{\frac{s\log(np)}{n}}$ (where the last step holds as long as $p\geq 2$). 
\end{proof}

\begin{lemma}\label{lem:bivariate_normal}
Let 
\[\left(\begin{array}{c}R\\S\end{array}\right)\sim \mathcal{N}\left(m , \left(\begin{array}{cc} r^2 & rs\cdot\rho \\ rs\cdot \rho&s^2\end{array}\right)\right),\]
and let $g,g_\star:\R\rightarrow\R$ be any non-decreasing functions, such that $g_\star$ is $L$-Lipschitz and satisfies $g_\star(T) - g_\star(-T)\geq C>0$ for some $T>0$. Then 
\[\EE{\big|g(R) - g_\star(S)\big|} \geq  C_0\sqrt{1-(\rho\vee 0)^2},\]
 where $C_0>0$ is a function of $L,T,C,s,|m_2|$ (not dependent on $\rho$), and is defined in the proof.
\end{lemma}
\begin{proof}[Proof of Lemma~\ref{lem:bivariate_normal}]

First, we will replace $R$ with $\tilde{R}=\frac{R-m_1}{r}$ and $S$ with $\tilde{S}=\frac{S-m_2}{s}$, to obtain
\[\left(\begin{array}{c}\tilde{R}\\\tilde{S}\end{array}\right)\sim \mathcal{N}\left(0, \left(\begin{array}{cc} 1 &\rho \\  \rho&1\end{array}\right)\right).\]
Define non-decreasing functions $\tilde{g}(t) = g(m_1+r\cdot t)$ and $\tilde{g}_\star(t) = g_\star(m_2+s\cdot t)$. Then we are looking to lower bound
\[\EE{\big|g(R) - g_\star(S)\big|}=\EE{\big|\tilde{g}(\tilde{R}) - \tilde{g}_\star(\tilde{S})\big|} .\]
Note that $\tilde{g}_\star$ is $\tilde{L}$-Lipschitz for $\tilde{L} = Ls$, and 
\[\tilde{g}_\star\left(\frac{T-m_2}{s}\right) - \tilde{g}_\star\left(\frac{-T-m_2}{s}\right) = g_\star(T) - g_\star(-T)\geq C,\]
and since $\tilde{g}_\star$ is nondecreasing, we can weaken this bound to
\[\tilde{g}_\star(\tilde{T}) - \tilde{g}_\star(-\tilde{T}) \geq C \text{ where }\tilde{T} = \frac{T+|m_2|}{s}.\]

First consider the case that $\rho\leq \rho_0$, where $\rho_0\in[0,1)$ is defined to satisfy $2Ls\sqrt{1-\rho_0^2}<\frac{1}{2}C\sqrt{2/\pi}$. In this case, define $c = \frac{\tilde g_\star(\tilde T)+\tilde g_\star(-\tilde T)}{2}$, then $\tilde g_\star(\tilde T)-c \geq \frac{C}{2}$ and $c - \tilde g_\star(-\tilde T)\geq \frac{C}{2}$. First suppose that $\tilde g(0)\leq c$. Then $\tilde g(r)\leq c$ for any $r\leq 0$ since $\tilde g$ is non-decreasing. Now, if $\tilde{S}\geq T$ and $\tilde{R}\leq 0$, this means that
\[ \tilde g_\star(\tilde{S}) \geq \tilde g_\star(T) \geq \frac{C}{2} + c \geq \frac{C}{2} + \tilde g(\tilde {R}),\]
and so $\big|\tilde g(\tilde{R} ) - \tilde g_\star(\tilde{S})\big| \geq \frac{C}{2}$. In other words,
\[\EE{\big|\tilde g(\tilde{R} ) - \tilde g_\star(\tilde{S})\big| }\geq\frac{C}{2}\cdot \PP{\tilde{S}\geq T,\tilde{R}\leq 0} .\]
Finally, 
\[\PP{\tilde{S}\geq T,\tilde{R}\leq 0} \geq \PP{\mathcal{N}\left(\left(\begin{array}{cc}0\\0\end{array}\right),\left(\begin{array}{cc}1&\rho_0\\\rho_0&1\end{array}\right)\right)\in [T,\infty)\times (-\infty,0]}\eqqcolon P_0(\rho_0,T)>0,\]
since for the first inequality, the probability is minimized at $\rho=\rho_0$ (under the assumption $\rho\leq \rho_0$), and is therefore lower bounded by some positive value $P_0 = P_0(T,\rho_0)$. 
Therefore, the conclusion of the lemma holds in this case, as long as we set $C_0 \leq \frac{C}{2}\cdot P_0(T,\rho_0)$.

Next we will work on the case that $\rho\geq \rho_0$.
 Let 
\[\check{S} = \rho\cdot \tilde{R} - \left(\tilde{S} - \rho\cdot \tilde{R}\right),\]
so that $(\tilde{R},\check{S})$ is equal in distribution to $(\tilde{R},\tilde{S})$.
Then
\[\EE{\big|\tilde{g}(\tilde{R}) - \tilde{g}_\star(\tilde{S})\big|} = \EE{\big|\tilde{g}(\tilde{R}) - \tilde{g}_\star(\check{S})\big|},\]
and so
\begin{align*}\EE{\big|\tilde{g}(\tilde{R}) - \tilde{g}_\star(\tilde{S})\big| }
&= \frac{1}{2}\EE{\Big(\big|\tilde{g}(\tilde{R}) - \tilde{g}_\star(\tilde{S})\big|+\big|\tilde{g}(\tilde{R}) - \tilde{g}_\star(\check{S})\big|\Big)}\\
&\geq \frac{1}{2}\EE{\big|\tilde{g}_\star(\tilde{S}) - \tilde{g}_\star(\check{S})\big|}\\
&=\frac{1}{2}\EE{\Big|\tilde{g}_\star\big(\tilde{S}\big) - \tilde{g}_\star\left(\frac{\tilde{S}+\check{S}}{2}\right)\Big| + \Big|\tilde{g}_\star\left(\frac{\tilde{S}+\check{S}}{2}\right)- \tilde{g}_\star\big(\check{S}\big)\Big|}\\
&=\frac{1}{2}\EE{\Big|\tilde{g}_\star\big(\tilde{S}\big) - \tilde{g}_\star\big(\rho\cdot \tilde{R}\big)\Big| + \Big|\tilde{g}_\star\big( \rho\cdot \tilde{R} \big) - \tilde{g}_\star\big(\check{S}\big)\Big|},
\end{align*}
where the third step holds since since $g_\star$ is non-decreasing.
Now let $W = \tilde{S} - \rho\cdot\tilde{R}\sim \mathcal{N}(0,1-\rho^2)$, and note $W\independent \tilde{R}$. Then we can rewrite this as 
\begin{align*}
&\frac{1}{2}\EE{\Big|\tilde{g}_\star\big(\tilde{S}\big) - \tilde{g}_\star\big(\rho\cdot \tilde{R}\big)\Big| + \Big|\tilde{g}_\star\big( \rho\cdot \tilde{R} \big) - \tilde{g}_\star\big(\check{S}\big)\Big|}\\
&=\frac{1}{2}\EE{\Big|\tilde{g}_\star\big(\rho\cdot \tilde{R} + W\big) - \tilde{g}_\star\big(\rho\cdot \tilde{R}-W\big)\Big| }\\
&=\frac{1}{2}\Ep{W}{\int_{t=-\infty}^{\infty}\phi(t)\cdot \Big|\tilde{g}_\star\big(\rho\cdot t + W\big) - \tilde{g}_\star\big(\rho\cdot t-W\big)\Big| dt }\text{ since $\tilde{R}\sim\mathcal{N}(0,1)$ is independent of $W$}\\
&\geq \frac{1}{2}\Ep{W}{\int_{t=-\tilde{T}/\rho}^{\tilde{T}/\rho}\phi(t)\cdot \Big|\tilde{g}_\star\big(\rho\cdot t + W\big) - \tilde{g}_\star\big(\rho\cdot t-W\big)\Big| dt}\text{ since the integrand is nonnegative}\\
&=\frac{1}{2}\Ep{W}{\int_{t=-\tilde{T}/\rho}^{\tilde{T}/\rho}\phi(t)\cdot \Big(\tilde{g}_\star\big(\rho\cdot t + |W|\big) - \tilde{g}_\star\big(\rho\cdot t-|W|\big)\Big) dt}\text{ since $\tilde{g}_\star$ is non-decreasing}\\
&=\frac{\phi(\tilde{T}/\rho)}{2}\Ep{W}{\int_{t=-\tilde{T}/\rho}^{\tilde{T}/\rho} \Big(\tilde{g}_\star\big(\rho\cdot t + |W|\big) - \tilde{g}_\star\big(\rho\cdot t-|W|\big)\Big) dt}\text{ since $\phi(t)\geq \phi(\tilde{T}/\rho)$ for all $|t|\leq \tilde{T}/\rho$}\\
&=\frac{\phi(\tilde{T}/\rho)}{2}\Ep{W}{\int_{t=-\tilde{T}/\rho}^{\tilde{T}/\rho} \tilde{g}_\star\big(\rho\cdot t + |W|\big) \;\mathsf{d}t- \int_{t=-\tilde{T}/\rho}^{\tilde{T}/\rho} \tilde{g}_\star\big(\rho\cdot t-|W|\big) dt}\\
&=\frac{\phi(\tilde{T}/\rho)}{2}\Ep{W}{\int_{t=-\tilde{T}/\rho+|W|/\rho}^{\tilde{T}/\rho+|W|/\rho} \tilde{g}_\star\big(\rho\cdot t \big) \;\mathsf{d}t- \int_{t=-\tilde{T}/\rho-|W|/\rho}^{\tilde{T}/\rho-|W|/\rho} \tilde{g}_\star\big(\rho\cdot t\big) dt} \\ &\text{ by a change of variables in each integral}\\
&=\frac{\phi(\tilde{T}/\rho)}{2}\Ep{W}{\int_{t=\tilde{T}/\rho-|W|/\rho}^{\tilde{T}/\rho+|W|/\rho} \tilde{g}_\star\big(\rho\cdot t \big) \;\mathsf{d}t - \int_{t=-\tilde{T}/\rho-|W|/\rho}^{-\tilde{T}/\rho+|W|/\rho} \tilde{g}_\star\big(\rho\cdot t\big) dt}\\ &\text{ by cancellation of the overlapping regions of integration}\\
&\geq \frac{\phi(\tilde{T}/\rho)}{2}\Ep{W}{\frac{2|W|}{\rho} \Big(\tilde{g}_\star(\tilde{T} - |W|) - \tilde{g}_\star(- \tilde{T} +|W|)\Big)}\text{ since $\tilde{g}_\star$ is non-decreasing}\\
&\geq\frac{\phi(\tilde{T}/\rho)}{2}\Ep{W}{\frac{2|W|}{\rho} \Big(\tilde{g}_\star(\tilde{T}) - \tilde{g}_\star(-\tilde{T}) - 2|W|\tilde{L}\Big)}\text{ since $\tilde{g}_\star$ is $\tilde{L}$-Lipschitz}\\
&\geq\frac{\phi(\tilde{T}/\rho)}{2}\Ep{W}{\frac{2|W|}{\rho} \Big(C- 2|W|\tilde{L}\Big)}\text{ since $\tilde{g}_\star({\tilde{T}}) - \tilde{g}_\star(-{\tilde{T}}) \geq C$}\\
&=\frac{1}{\rho}\phi(\tilde{T}/\rho)\left(C\EE{|W|}- 2\tilde{L}\EE{W^2}\right)\\
&=\frac{1}{\rho}\phi(\tilde{T}/\rho)\left(C\sqrt{2/\pi}\sqrt{1-\rho^2}- 2\tilde{L}(1-\rho^2)\right)\text{ since $W\sim\mathcal{N}(0,1-\rho^2)$.}
\end{align*}
Plugging in our definitions of $\tilde{T}$ and $\tilde{L}$ and combining everything,
\[\EE{\big|g(R) - g_\star(S)\big|}\geq \frac{1}{\rho}\cdot \phi\left(\frac{T+|m_2|}{s\rho}\right)\cdot \left(C\sqrt{2/\pi}\sqrt{1-\rho^2}- 2Ls(1-\rho^2)\right).\]
Now, recall that $\rho_0$ satisfies $2Ls\sqrt{1-\rho_0^2}<\frac{1}{2}C\sqrt{2/\pi}$. Then since $\frac{1}{\rho}\geq 1 \geq \rho_0$,
\[\EE{\big|g(R) - g_\star(S)\big|}\geq \sqrt{1-\rho^2}\cdot \rho_0\cdot \phi\left(\frac{T+|m_2|}{s\rho_0}\right)\cdot \left(C\sqrt{2/\pi}- 2Ls\sqrt{1-\rho_0^2}\right).\]
Therefore, as long as $C_0 \leq \rho_0\cdot \phi\left(\frac{T+|m_2|}{s\rho_0}\right)\cdot \left(C\sqrt{2/\pi}- 2Ls\sqrt{1-\rho_0^2}\right)$, in both cases ($\rho\geq \rho_0$ and $\rho<\rho_0$), we have proved that $\EE{\big|g(R) - g_\star(S)\big|}\geq C_0\sqrt{1-\rho^2}$ in the case that $\rho\geq \rho_0$, as desired.
\end{proof}

\end{document}